\documentclass[final,1p,times]{elsarticle}
\usepackage{amsmath}

\usepackage{tikz}
\usepackage{amsmath}

\usetikzlibrary{arrows.meta,positioning,shapes.multipart}
\usepackage{verbatim}  % 需要添加到导言区
\allowdisplaybreaks[4]
\usepackage{amsthm}
\usepackage{amscd}
\RequirePackage{CJKnumb}
\usepackage{amsfonts}
\usepackage{amssymb}
\usepackage{graphicx}
\usepackage{subcaption} % 用于子图
\usepackage{caption}

\usepackage{booktabs}
\biboptions{sort&compress}
\usepackage{xcolor}
\usepackage[pagebackref, colorlinks, citecolor=green, linkcolor=red, urlcolor=blue]{hyperref}
\usepackage{pgfplots}
\pgfplotsset{compat=newest} % 使用最新版本的 pgfplots
\numberwithin{equation}{section}

\usepackage{mathrsfs}
\usepackage{titletoc}
\usepackage{float}
\usepackage{longtable}
\usepackage{adjustbox}
\newcommand{\be}{\begin{equation}}
	\newcommand{\ee}{\end{equation}}	
\theoremstyle{remark} 

\journal{***}
\begin{document}
	\begin{frontmatter}
			\title{Stability and Hopf bifurcation analysis of an age-structured SVIRS epidemic model with temporary immunity}
		
		\author{Songbo Hou \corref{cor1}}
		\ead{housb@cau.edu.cn}
		\address{Department of Applied Mathematics, College of Science, China Agricultural University,  Beijing, 100083, P.R. China}
		\author{Xinxin Tian}
		\ead{txx@cau.edu.cn}
		\address{Department of Applied Mathematics, College of Science, China Agricultural University,  Beijing, 100083, P.R. China}
		
		\cortext[cor1]{Corresponding author: Songbo Hou}
		\begin{abstract}
	
		In this paper, we investigate an SVIRS epidemic model that incorporates both temporary immunity and an age-structured recovery process. By reformulating the system as a non-densely defined abstract Cauchy problem, we establish the existence and uniqueness of solutions and derive the basic reproduction number \( \mathcal{R}_0 \). The stability of the equilibria is analyzed through the associated characteristic equations, and the occurrence of Hopf bifurcation near the endemic equilibrium is rigorously demonstrated. Our theoretical results reveal that temporary immunity plays a crucial role in shaping the stability of the endemic state. Finally, numerical simulations are carried out to verify and illustrate the analytical findings.
		
		\end{abstract}	
	\end{frontmatter}
\section{Introduction}
The worldwide dissemination of communicable diseases now represents one of the most pressing concerns confronting modern public health systems, exerting profound impacts on multiple dimensions of human society. Infectious diseases not only constitute an immediate peril to public health and mortality but also severely disrupt the normal functioning of society. For instance, in the economic sector, they lead to the shutdown of enterprises and the interruption of supply chains; in the educational field, they result in the suspension of schools and the restriction of teaching activities. Additionally, they affect people's mental health, among other aspects. Given the numerous harms that infectious diseases bring to human society, it is imperative to implement prevention and control measures against them under this backdrop.

Vaccination is among the most efficient ways to prevent and keep infectious diseases under control. Its principle lies in stimulating the body to produce a specific immune response by administering vaccines containing attenuated or inactivated pathogen components. In 1927, Kermack and McKendrick \cite{Kermack1927} proposed the SIR model. In this model, each variable corresponds to an epidemiological category: S (susceptible), I (infected), and R (recovered). This framework laid the foundation for subsequent infectious disease modeling research. Subsequently, scholars have developed improved models like SIS and SEIR \cite{ZAMAN200943}. Based on this framework, the impact of vaccination can be quantitatively analyzed by adding corresponding compartments (such as vaccinated individuals $V$) to the basic epidemiological models.

In recent years, many researchers have performed extensive studies on epidemiological vaccination models \cite{DUAN2014528,10,Goel2020,LIU20081,SONG2022106011,KRIBSZALETA2000183,SUN2023113206}. Kribs-Zaleta et al. \cite{KRIBSZALETA2000183} introduced a vaccination compartment V into the SIS model to study prevention and control strategies for diseases like pertussis and tuberculosis. A continuous age-structured SVIR model was formulated by Wang et al. \cite{10}, focusing on factors such as susceptibility, vaccination effectiveness, and disease recurrence. Duan et al. \cite{DUAN2014528} incorporated the age variable of vaccination into the SVIR model to systematically analyze the effects of vaccine efficacy and vaccination age on disease transmission. Liu et al. \cite{LIU20081} demonstrated that optimizing vaccination strategies could achieve disease elimination by constructing a dual-strategy SVIR model. Song et al. \cite{SONG2022106011} proposed an SEAIQR model that includes vaccination and quarantine delays, while Sun et al. \cite{SUN2023113206} studied an age-structured SVIR model. Zhu et al. \cite{eltit} postulated uniform natural mortality rates across age groups and examined vaccination effects, constructing the subsequent SVIR structure:
	\begin{equation}\label{1.1}
\begin{aligned}\left\{\begin{aligned}
		&\frac{dS(t)}{dt} = \Pi-\beta I(t)S(t) - (\mu+\iota) S(t)+ \eta V(t), \\
		&\frac{dV(t)}{dt}=\iota S(t)-(1-\epsilon)\beta I(t)V(t)-(\mu + \eta) V(t), \\
		&\frac{dI(t)}{dt}= \beta I(t) (S(t) + (1-\epsilon) V(t)) - (\mu + \gamma + d) I(t), \\
		&\frac{dR(t)}{dt}=\gamma I(t)-\mu R(t),\\
	\end{aligned}\right.\end{aligned}
\end{equation}
where $\Pi$ denotes the recruitment rate of susceptible individuals, $\beta$ denotes the effective transmission rate from interactions involving susceptible and infected individuals, $\mu$ corresponds to the baseline mortality rate shared across all compartments, $\epsilon\in [0,1]$ represents vaccine efficacy (with $\epsilon=0$ indicating completely ineffective vaccination and $\epsilon=1$ implying perfect protection), the vaccination rate for susceptible individuals is denoted by $\iota$, $d$ denotes the mortality rate induced by infection, $\gamma$ stands for the rate at which individuals recover from infection, with $ \eta $ denoting the immunity waning rate, leading vaccinated individuals back to susceptibility.

Age, as one of the vital biological indicators, holds a key position in the research of infectious disease modeling. Researchers have developed age-structured models that utilize partial differential equations (PDEs) and ordinary differential equations (ODEs) by treating age as a continuous variable. Such models are capable of more precisely depicting the dynamic characteristics of disease transmission. In relevant research, scholars like Webb et al. have conducted systematic investigations into age-structured models \cite{Webb2008,DUCROT2008501,Iannelli1994}. Numerous empirical investigations have demonstrated that these models offer significant advantages in studying the transmission dynamics of epidemics \cite{CHEN201616,LIU201518}.

Theoretical advancements in age-structured epidemic models have achieved major milestones. The establishment of central manifold theory was accomplished by Magal and Ruan \cite{center}, while simultaneously, Liu et al. \cite{Liu2011} established a theorem that formulates Hopf bifurcation conditions for non-densely defined Cauchy problems. These works laid a crucial mathematical foundation for subsequent research. Building upon this foundation, Wang and Liu \cite{WANG20121134} successfully achieved an effective analysis of the local Hopf bifurcation phenomenon in age-structured host-pathogen systems by transforming them into non-dense Cauchy problems. The modelling framework has expanded beyond infection age to incorporate recovery age \cite{DUAN2017613} and vaccination age structures \cite{YAN2025104310, doi:10.1142/S0218127424501967}.
\begin{figure}
	\centering
	\includegraphics[width=0.7\textwidth]{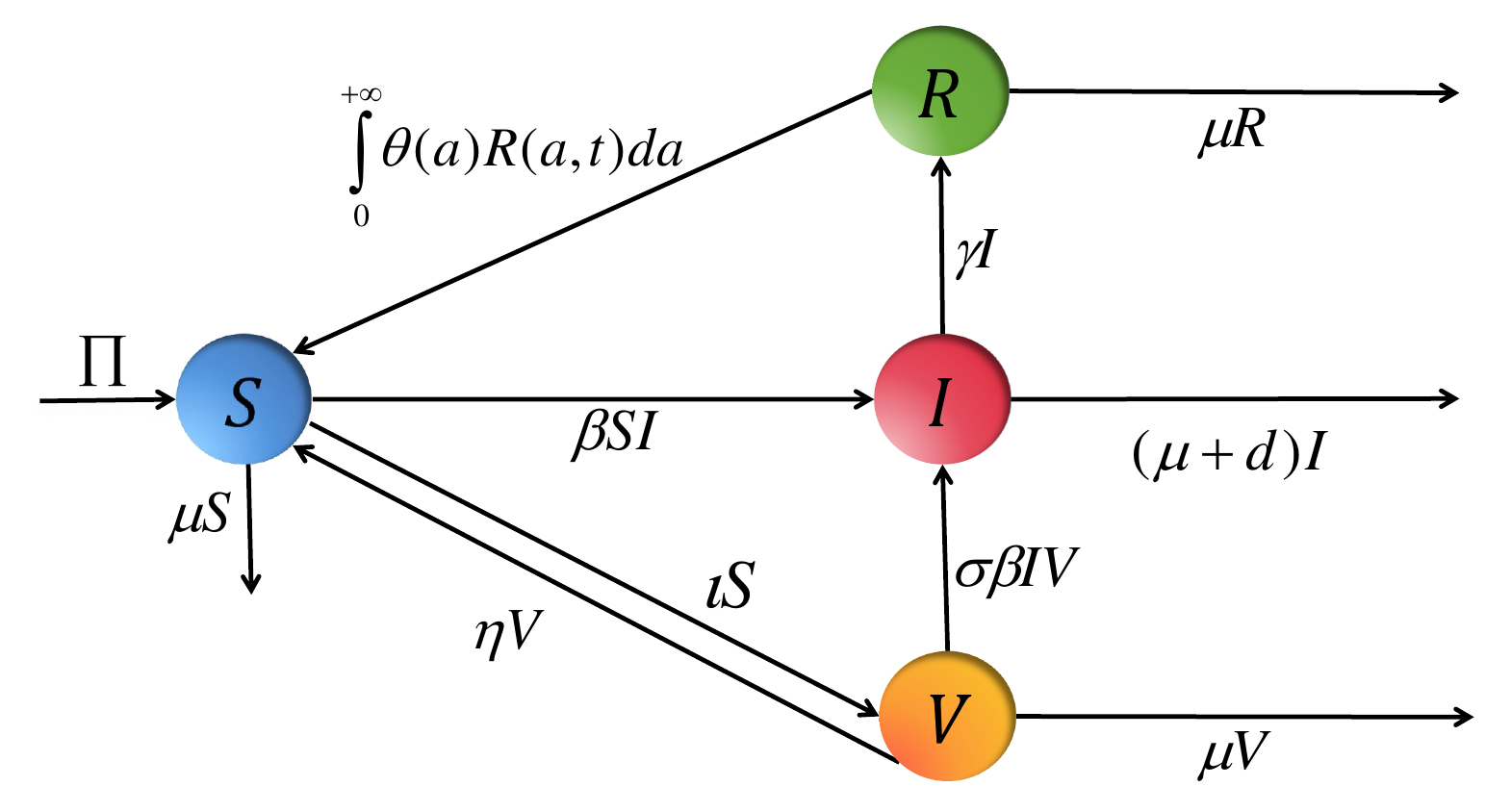}
	\caption{The pictorial representation of proposed model}
	\label{fig:00}
\end{figure}

Substantial progress has been made in modeling immunological dynamics. Epidemiological studies demonstrate that acquired immunity in recovered individuals exhibits temporal decay, leading to reversion to susceptibility. For instance, influenza infection confers durable strain-specific immunity but limited heterosubtypic protection \cite{KYRYCHKO2005495}. Two principal methodologies model transient immunity: (i) fixed-duration immunity through time-delay formulations \cite{BAI2025109413} and (ii) age-dependent immune waning processes. Implementing the latter approach, Duan et al. \cite{DUAN2017613} established the exact conditions required for Hopf bifurcation to occur in an age-structured immunity SIRS model. Zhang et al. \cite{doi:10.1142/S0218127421501832} extended earlier modeling frameworks by introducing an age-structured approach that accounts for both quarantine measures and short-term immune responses. Comprehensive reviews of age-structured modeling developments are available in references \cite{Jiang2023,Guo01012018,Tan2020,Lin01012019,zhang,Wang2023,CAI2017127}.

These findings motivate a substantial extension of model \eqref{1.1} in which recovery age $a$ is incorporated into the recovered individuals, resulting in an SVIRS system that features immunity duration dependence:
\begin{equation}\label{1.2}
	\begin{aligned}\left\{\begin{aligned}
			&\frac{dS(t)}{dt} = \Pi-\beta I(t)S(t) - (\mu+\iota) S(t)+ \eta V(t)+\int_{0}^{+\infty} \theta(a) R(a,t)da, \\
			&\frac{dV(t)}{dt}=\iota S(t)-\sigma\beta I(t)V(t)-(\mu + \eta) V(t), \\
			&\frac{dI(t)}{dt}= \beta I(t) (S(t) + \sigma V(t)) - (\mu + \gamma + d) I(t), \\
			&(\frac{\partial}{\partial t} + \frac{\partial}{\partial a})R(a,t)=-(\mu+\theta(a)) R(a,t),\\
		\end{aligned}\right.\end{aligned}
\end{equation}
subject to the initial and boundary conditions
\begin{equation}\label{1.3}
\begin{cases} 
		S(0) = S_0, V(0)=V_0, I(0) = I_0,R(a,0) = R_0(a),\\
	R(0,t) = \gamma I(t),
\end{cases}
\end{equation}
in which \( \sigma\in[0,1] \) stands for the efficiency of the vaccine, while $\theta(a)\in L_+^{\infty}((0,+\infty),\mathbb{R})$ signifies the rate at which acquired immunity in recovered individuals diminishes as age increases, satisfying the following assumptions. In order to more accurately depict this phenomenon, we have included a schematic diagram of the model \eqref{1.2} in Figure 1.
\theoremstyle{italiclemma}
\newtheorem{assumption}{\bf Assumption}[section]
\begin{assumption}
	Suppose that
	$$\theta\left(a\right)=\left\{
	\begin{array}
		{cc}0, & \quad0<a<\tau, \\
		\vartheta^{*}, & \quad a\geq\tau,
	\end{array}\right.$$
	where the immunity period is represented by $\tau$ and the constant $\vartheta^*\in \mathbb{R}^+$.
\end{assumption}
The following section sequence is adopted: Section 2 reconstructs system \eqref{1.2} as a non-densely defined abstract Cauchy problem and establishes solution existence and uniqueness. Section 3 derives the basic reproduction number $\mathcal{R}_0$, proves equilibrium point existence, and provides linearization outcomes. Sections 4 and 5 respectively address stability conditions of equilibrium points and Hopf bifurcation existence. Section 6 validates theoretical findings through numerical simulations, with the final section summarizing research contributions.
\section{Preliminaries}
Let the Banach space $X$ be defined by the product $\mathbb{R}^3\times L^1((0,+\infty),\mathbb{R})\times\mathbb{R}$, equipped with the norm
$$\|\phi\|=\|\phi_{1}\|_{\mathbb{R}}+\|\phi_{2}\|_{\mathbb{R}}+\|\phi_{3}\|_{\mathbb{R}}+\|\phi_{4}\|_{L^{1}((0,+\infty),\mathbb{R})}+\|\phi_{5}\|_{\mathbb{R}},\quad\forall\phi\in X.$$
Consider the linear operator  $B:D(B)\subset X\to X$ given by
$$B\phi=\begin{pmatrix}-(\mu + \iota)\phi_1  \\ 
	- (\mu + \eta)\phi_2 \\ 
	-(\mu + \gamma + d)\phi_3\\ 
	-\dfrac{d\phi_4}{da}-(\mu+\theta(a))\phi_4\\
	-\phi_4(0)\end{pmatrix},$$
with its domain defined as
$$D(B)=\mathbb{R}^3\times W^{1,1}((0,+\infty),\mathbb{R})\times\{0\}.$$
Since $$\overline{D(B)} =\mathbb{R}^3\times L^1((0,+\infty), \mathbb{R}) \times \{0\} =: X_0,$$ it follows that $X_0$ is a non-dense subset of $X$.\\
Define a nonlinear operator $F:X_{0}\to X$ as
$$F\phi=\left(\begin{array}{c}\Pi-\beta \phi_1\phi_3+\eta \phi_2+\int_{0}^{+\infty} \theta(a)\phi_4(a) da\\
	\iota \phi_1 -\sigma\beta \phi_3\phi_2\\
	\beta \phi_1\phi_3+\sigma\beta \phi_3\phi_2\\
	0\\
	\gamma\phi_3\end{array}\right),$$
where $\phi=(\phi_{1},\phi_{2},\phi_{3},\phi_{4},0)^{T}\in X_{0}$. With the state variable $r(t) = (S(t), V(t), I(t), R(a,t), 0)^{T} \in X$, and adopting the operator definitions $B$ and $F$ from above, system \eqref{1.2} admits the following abstract Cauchy formulation:
\begin{equation}\label{2.1}
\begin{cases}\dfrac{dr(t)}{dt}=B(r(t))+F(r(t)),&t>0,\\r(0)=r_{0},\end{cases}
\end{equation}
where $r_{0}=(S_{0},V_{0},I_{0},R_{0}(a),0)^{T}\in X$.
Set
$$\chi := \min \{\mu , \iota,\eta,\gamma+d \} > 0 \quad \text{and} \quad\Omega := \{ v \in C : \operatorname{Re}(v) > - \chi \} .$$
The results established in Reference \cite{article} yield the following lemma.
\vskip 0.2cm
\newtheorem{lemma}{Lemma}[section]  % 定义引理环境
\begin{lemma}
	Operator $B$ qualifies as a Hille–Yosida operator on $X$.
\end{lemma}
\begin{proof}
If $(\phi_1,\phi_2,\phi_3,\phi_4,\phi_5)\in X$,	
$$(v I-B)^{-1}
\begin{pmatrix}
	\phi_1 \\
	\\
	\phi_2 \\
	\\
	\phi_3 \\
	\\
	\phi_4 \\
	\\
	\phi_5
\end{pmatrix}=
\begin{pmatrix}
	\tilde{\phi}_1 \\
	\\
	\tilde{\phi}_2 \\
	\\
	\tilde{\phi}_3 \\
	\\
	\tilde{\phi}_4 \\
	\\
	0
\end{pmatrix},$$
that is
$$
\begin{cases}
	(v+\mu+\iota)\tilde{\phi}_1=\phi_1, \\
	(v+\mu+\eta)\tilde{\phi}_2=\phi_2,\\
	(v+\mu+d+\gamma)\tilde{\phi}_3=\phi_3, \\
	\tilde{\phi}_4^{\prime}=-(v+\mu+\theta(a))\tilde{\phi}_4+\phi_4, \\
	\tilde{\phi}_4(0)=\phi_5. & 
\end{cases}$$
As a result, we get
\begin{equation}\label{2.2}
	\begin{cases}
		\tilde{\phi}_1 = \dfrac{\phi_1}{v + \mu + \iota}, \\
		\tilde{\phi}_2 = \dfrac{\phi_2}{v + \mu + \eta}, \\
		\tilde{\phi}_3 = \dfrac{\phi_3}{v + \mu + d + \gamma}, \\
		\tilde{\phi}_4(a) = e^{-\int_{0}^{a}(v+\mu+\theta(s))\,\mathrm{d}s} \phi_5 
		+ \int_{0}^{a} e^{-\int_{s}^{a}(v+\mu+\theta(s'))\,\mathrm{d}s'} \phi_4(s)\,\mathrm{d}s.
	\end{cases}
	\end{equation}
Applying integration to the fourth equation of system \eqref{2.2} yields 
\begin{equation}\label{2.3}\|\tilde{\phi}_{4}\|_{L^{1}} \leq \frac{1}{\operatorname{Re}(v)+\chi}\left(\|\phi_{4}\|_{L^{1}}+|\phi_{5}|\right).
\end{equation}	
From \eqref{2.2} and \eqref{2.3}, we get
$$
	|\tilde{\phi}_{1}|+|\tilde{\phi}_{2}|+|\tilde{\phi}_{3}|+\|\tilde{\phi}_{4}\|_{L^{1}} \le\frac{1}{\operatorname{Re}(v)+\chi}\left(|\phi_{1}|+|\phi_{2}|+|\phi_{3}|+\|\phi_{4}\|_{L^{1}}+|\phi_{5}|\right).
$$
Hence, we have
$$\left\|(v I -B)^{-1}\right\| \leq \frac{1}{\operatorname{Re}(v)+ \chi}, \quad \text{for all } v \in \Omega.$$
This implies that the operator $B$ qualifies as a Hille–Yosida operator.
\end{proof}
Set
$$X_0^+=\mathbb{R}_+^3\times L_+^1((0,+\infty),\mathbb{R})\times\{0\}.$$
The well-posedness of solutions to \eqref{2.1} can be verified using the theory in \cite{2001,Magal2018intro}.
\vskip 0.2cm
\newtheorem{thm}{\bf Theorem}[section]
\begin{thm}
 A unique continuous semi-flow solution $\{ T(t)\}_{t\geq 0}$ exists on $X_{0_{+}}$, such that for any $r(0) \in X_0^{+}$, the function $t \to T(t)r(0)$ represents the unique integrated solution of system \eqref{2.1}, which can be expressed as $$T(t)r(0) = r(0) + B\int_0^t T(s)r(0)\mathrm{d}s + \int_0^t F(T(s)r(0))\mathrm{d}s, \quad \forall\: t \geq 0.$$

Furthermore, on account of $B$ as the Hille Yosida operator, it yields a nondegenerate integrated semi-group $\{T_B(t)\}_{t\geq0}$ on $X.$ Introduce its part $B_{0}$, i.e.
$$B_{0}\phi=B\phi,\quad D(B_{0})=\{\phi\in D(B):B\phi\in X_{0}\}.$$
Then $B_{0}$ produces a $C_{0}$-semi-group $\{T_{B_{0}}(t)\}_{t\geq0}$ on $X_{0}$.
\end{thm}

From a biological standpoint, we restrict our analysis to non-negative solutions. The following results establish that system \eqref{1.2} preserves non-negativity and boundedness for all non-negative initial conditions.
\begin{thm}
Any positive initial condition \( (S(0), V(0), I(0), R(a,0))^T \) that fulfills the requirement from \eqref{1.3} ensures the solution to system \eqref{1.2} is non-negative for all \( t \geq 0 \).  
\end{thm}
\begin{proof}
Denote by $l(t) = \min\{S(t), V(t), I(t)\}$ for each $t \ge 0$. The positivity of the initial values then yields $l(0) > 0$. Our objective is to verify that \( l(t) \ge 0 \) for every \( t \ge 0 \).  To this end, we argue by contradiction. Assume that there exists $t_1>0$ such that $l(t)>0$ for all $0 \leq t<t_1$, while $l\left(t_1\right)=0$. Through examining \( l(t_1) \), we consider the following cases:

(1) Given that $l(t_1) = S(t_1) = 0$, the initial equation from system \eqref{1.2} yields:
$$
\begin{aligned}
	\frac{dS(t)}{dt} &= \Pi - \beta I(t) S(t) - (\mu+\iota)S(t) + \eta V(t) + \int_{0}^{+\infty} \theta(a) R(a,t) \, da \\
	&\geq -\max_{0\leq t\leq t_1}\left\{\beta I(t)\right\}S(t) - (\mu+\iota)S(t) \\
	&= -b_1 S(t),
\end{aligned}
$$
for $t \in [0, t_1]$, where $b_1 = \max_{0\leq t\leq t_1}\left\{\beta I(t)\right\} + (\mu+\iota)$. Thus, $S(t_1) \geq S(0)e^{-b_1 t_1} > 0$, contradicting $S(t_1) = 0$.  

(2) If $l(t_1) = V(t_1) = 0$, the second equation implies: \\
$$
\begin{aligned}
	\frac{dV(t)}{dt}&= \iota S(t) - \sigma\beta I(t)V(t) - (\mu + \eta) V(t)\\
	&\geq -\max_{0\leq t\leq t_1}\left\{\sigma\beta I(t)\right\}V(t) - (\mu + \eta) V(t) \\
	&= -b_2 V(t),
\end{aligned}
$$
for $t \in [0, t_1]$, where $b_2 = \max_{0\leq t\leq t_1}\left\{\sigma\beta I(t)\right\} + (\mu + \eta)$. Hence, $V(t_1) \geq V(0)e^{-b_2 t_1} > 0$, contradicting $V(t_1) = 0$.  

(3) If $l(t_1) = I(t_1) = 0$, the third equation gives:  \\
$$
\begin{aligned}
	\frac{dI(t)}{dt} &= \beta I(t) (S(t) + \sigma V(t)) - (\mu + \gamma + d) I(t) \\
	&\geq - (\mu + \gamma + d) I(t) \\
	&= -b_3 I(t),
\end{aligned}
$$
for $t \in [0, t_1]$, where $b_3 = \mu + \gamma + d$. Therefore, $I(t_1) \geq I(0)e^{-b_3 t_1} > 0$, contradicting $I(t_1) = 0$.  

This demonstrates the non-negativity of $S(t)$, $V(t)$, and $I(t)$.  
Additionally, integrating the fourth equation of system \eqref{1.2} along the characteristics: 
$$
R(a,t) = 
\begin{cases}
	R(0,t-a)e^{-\int_{0}^{a} (\mu+\theta(s))\, ds}, & a \leq t, \\
	R_0(a-t)e^{-\int_{0}^{t}(\mu+\theta(a-t+s))\, ds}, & a > t,
\end{cases}
$$  
which shows that $R(a,t) \geq 0$ for nonnegative initial values.
\end{proof}
\begin{thm}
For every positive initial condition \( (S(0), V(0), I(0), R(a,0))^{T} \) that fulfills the initial constraints from \eqref{1.3}, system \eqref{1.2} is always bounded.   
\end{thm}
\begin{proof}
 Define \( W(t) = S(t) + V(t) + I(t) + \int_0^{+\infty} R(a, t)da \). Through system \eqref{1.2}, it follows that
	$$
	\begin{aligned}
		W'(t) &= S'(t) + V'(t) + I'(t) + \frac{d}{dt} \left( \int_0^{+\infty} R(a, t) \, da \right) \\
		&= \Pi - \beta I(t) S(t) - (\mu + \iota) S(t) + \eta V(t) + \int_0^{+\infty} \theta(a) R(a, t) \, da \\
		&\quad + \iota S(t) - \sigma \beta I(t) V(t) - (\mu + \eta) V(t) \\
		&\quad + \beta I(t) (S(t) + \sigma V(t)) - (\mu + \gamma + d) I(t) \\
		&\quad + \int_0^{+\infty} \left( -\frac{\partial R(a, t)}{\partial a} - (\mu + \theta(a)) R(a, t) \right) da \\
		&= \Pi - \mu S(t) - \mu V(t) - (\mu + \gamma + d) I(t) + \int_0^{+\infty} \left( -\frac{\partial R(a, t)}{\partial a} - \mu R(a, t) \right) da.
	\end{aligned}
	$$	
	With respect to the integral term, observe that:
	$$
	\int_0^{+\infty} \left( -\frac{\partial R(a, t)}{\partial a} - \mu R(a, t) \right) da = - \int_0^{+\infty} \frac{\partial R(a, t)}{\partial a} da - \mu \int_0^{+\infty} R(a, t) da.
	$$
	Assuming $ R(a, t) \to 0 $ sufficiently fast as $ a \to +\infty $, we have
	$$
	\int_0^{+\infty} \frac{\partial R(a, t)}{\partial a} da = R(+\infty, t) - R(0, t) = -R(0, t) = -\gamma I(t).
	$$
	Thus,
	$$
	\int_0^{+\infty} \left( -\frac{\partial R(a, t)}{\partial a} - \mu R(a, t) \right) da = \gamma I(t) - \mu \int_0^{+\infty} R(a, t) da.
	$$	
	
	Substituting back, we obtain	
$$	\begin{aligned}
		W'(t) &= \Pi - \mu S(t) - \mu V(t) - (\mu + \gamma + d) I(t) + \gamma I(t) - \mu \int_0^{+\infty} R(a, t) da\\
		&= \Pi - \mu \left( S(t) + V(t) + I(t) + \int_0^{+\infty} R(a, t) da \right) - d I(t) \\
		&=\Pi - \mu W(t) - d I(t) \\
		&\leq \Pi - \mu W(t).
	\end{aligned}
	$$
	
This indicates that 
$$ \limsup_{t \to +\infty} W(t) \leq \frac{\Pi}{\mu}.$$
The asymptotic behavior of system \eqref{1.2} is restricted to the biologically permissible domain:
$$ \Psi = \left\{ (S, V, I, R(\cdot)) \in X_0 : S, V, I, R(\cdot) \geq 0, \, S + V + I + \int_0^{+\infty} R(a, t) da \leq \frac{\Pi}{\mu} \right\},$$ 
which proves the ultimate boundedness of the system. 
\end{proof}
\section{Stability of equilibria}
In this section, we first show that system \eqref{1.2} admits an equilibrium point and then derive its linearization.
\subsection{Existence of equilibria}
We now utilize the subsequent theorem to demonstrate the presence of an equilibrium for system \eqref{1.2}.
\begin{thm}
For system \eqref{1.2}, there exists a disease-free equilibrium \( E^0 = (S^0, V^0, 0, 0) \), in which \( S^0 = \frac{\Pi(\mu+\eta)}{\mu(\mu+\iota+\eta)} \) and \( V^0=\frac {\iota\Pi}{\mu(\mu+\iota+\eta)} \). 
\end{thm}
\begin{proof}
We remark that the equilibria of system \eqref{1.2} are time-independent. Taking the representation $(S,V,I,R(a))$, these solutions satisfy the time-independent equations: 
\begin{equation}\label{3.1}
	\begin{cases}
		\Pi - \beta I S - (\mu+\iota) S + \eta V + \int_0^{+\infty} \theta(a) R(a) \, da = 0, \\
		\iota S - \sigma \beta I V - (\mu + \eta) V = 0, \\
		\beta I (S + \sigma V) - (\mu + \gamma + d) I = 0, \\
		\dfrac{dR(a)}{da} = -(\mu + \theta(a)) R(a), \\
		R(0) = \gamma I.
	\end{cases}
\end{equation}
	For the disease-free equilibrium $E^0 = (S^0, V^0, 0, 0)$, substitution of $I^0 = 0$ and $R^0(a) = 0$ reduces the system \eqref{3.1} to:
$$
\Pi - (\mu+\iota) S^0 + \eta V^0 = 0, \quad \iota S^0 - (\mu+\eta) V^0 = 0.
$$
Solving these yields:
$$
S^0 = \frac{\Pi(\mu+\eta)}{\mu(\mu+\iota+\eta)}, \quad V^0 = \frac{\iota \Pi}{\mu(\mu+\iota+\eta)}.
$$
As a consequence, the disease-free equilibrium is unique and exists.
\end{proof}
\begin{thm}
	Under the conditions $\mathcal{R}_0 > 1$ and $J(\tau)<\mu+\gamma+d$, system \eqref{1.2} possesses a unique endemic equilibrium $E^* = (S^*, V^*, I^*, R^*(a))$, where the parameters are defined as $$\mathcal{R}_0 = \frac{\Pi\beta(\mu+\eta+\sigma\iota)}{\mu(\mu+\iota+\eta)(\mu+\gamma+d)},\qquad J(\tau)=\frac{\vartheta^*\gamma e^{-\mu\tau}}{\mu+\vartheta^*}.$$
\end{thm}
\begin{proof}
	We first determine the equilibrium states of system \eqref{1.2} and then derive the basic reproduction number by employing the next-generation approach \cite{van2002reproduction}. We begin by defining the new infection term $\mathcal{F}$ and the transition term $\mathcal{V}$ in the following manner: 
$$
	\mathcal{F} = \beta I(S + \sigma V), \quad \mathcal{V} = (\mu + \gamma + d)I.
	$$
 Considering the infection-free equilibrium $E^0 = (S^0, V^0, 0, 0)$, we derive the corresponding Jacobian matrices (which reduce to scalars here due to a single infection variable $ I$). 
	$$
	F = \left.\frac{\partial \mathcal{F}}{\partial I}\right|_{E^0} = \beta(S^0 + \sigma V^0), \quad V = \left.\frac{\partial \mathcal{V}}{\partial I}\right|_{E^0} = \mu + \gamma + d.
	$$
	The basic reproduction number is then given by
	$$
	\mathcal{R}_0 = \rho(FV^{-1}) = \frac{\beta(S^0 + \sigma V^0)}{\mu + \gamma + d} = \frac{\Pi\beta(\mu+\eta+\sigma\iota)}{\mu(\mu+\iota+\eta)(\mu+\gamma+d)}.
$$	
The endemic equilibrium \( E^* \) has the form \( (S^*, V^*, I^*, R^*(a)) \) and satisfies:
	$$
	\begin{cases}
		\Pi - \beta I^*S^* - (\mu + \iota)S^* + \eta V^* + \int_{0}^{+\infty} \theta(a)R^*(a)da = 0, \\
		\iota S^* - \sigma \beta I^*V^* - (\mu + \eta)V^* = 0, \\
		\beta(S^* + \sigma V^*) - (\mu + \gamma + d) = 0, \\
		\frac{dR^*(a)}{da} = -(\mu + \theta(a))R^*(a), \\
		R^*(0) = \gamma I^*.
	\end{cases}
	$$
The third equation produces the relation:
	$$
	S^* + \sigma V^* = \frac{\mu + \gamma + d}{\beta}.
	$$
	From the second equation, we express $V^*$ in terms of $S^*$:
	$$
	V^* = \frac{\iota S^*}{\mu + \eta + \sigma \beta I^*}.
	$$
	Substituting into the expression for $S^* + \sigma V^*$ yields:
	$$
	S^* \left(1 + \frac{\sigma \iota}{\mu + \eta + \sigma \beta I^*}\right) = \frac{\mu + \gamma + d}{\beta},
	$$
	which gives:
	$$
	S^* = \frac{\mu + \gamma + d}{\beta} \cdot \frac{\mu + \eta + \sigma \beta I^*}{\mu + \eta + \sigma \beta I^* + \sigma \iota}.
	$$
	
	The solution to the recovery equation with boundary condition $R^*(0) = \gamma I^*$ is:
	$$
	R^*(a) = \gamma I^* e^{-\int_0^a (\mu + \theta(s)) ds}.
	$$
	Using the piecewise definition of $\theta(a)$, we evaluate the integral:
	$$
	\int_{0}^{a}(\mu+\theta(s))ds = \begin{cases}
		\mu a, & 0 < a < \tau, \\
		\mu \tau + (\mu + \vartheta^{*})(a - \tau), & a \geq \tau,
	\end{cases}
	$$
	resulting in:
	$$
	R^*(a) = \begin{cases}
		\gamma I^* e^{-\mu a}, & 0 < a < \tau, \\
		\gamma I^* e^{-\mu \tau - (\mu + \vartheta^{*})(a - \tau)}, & a \geq \tau.
	\end{cases}
	$$
	
	The integral term in the first equation becomes
	$$
	\int_{0}^{+\infty} \theta(a) R^*(a) da = \vartheta^* \gamma I^* e^{-\mu \tau} \int_{\tau}^{+\infty} e^{-(\mu + \vartheta^*)(a - \tau)} da = \frac{\vartheta^* \gamma I^* e^{-\mu \tau}}{\mu + \vartheta^*} := J(\tau)I^*,
	$$
where $	J(\tau)=\dfrac{\vartheta^* \gamma  e^{-\mu \tau}}{\mu + \vartheta^*}.$

	Plugging $S^*$, $V^*$, and $J(\tau)$ into the initial equation yields:  
	$$\begin{aligned}
			f(I^*) :=&\Pi - \frac{(\mu + \gamma + d)(\mu + \eta + \sigma \beta I^*)I^*}{\mu + \eta + \sigma \beta I^* + \sigma \iota} - \frac{(\mu + \iota)(\mu + \gamma + d)(\mu + \eta + \sigma \beta I^*)}{\beta(\mu + \eta + \sigma \beta I^* + \sigma \iota)} \\
			&+ \frac{\iota \eta (\mu + \gamma + d)}{\beta(\mu + \eta + \sigma \beta I^* + \sigma \iota)} + J(\tau) I^* = 0.
	\end{aligned}$$

	We analyze $f(I^*)$ at the boundaries:
	\begin{itemize}
		\item At $I^* = 0$:
		$$
		f(0) =\Pi \left(1 - \frac{1}{\mathcal{R}_0}\right) > 0 \quad \text{when} \quad \mathcal{R}_0 > 1.
		$$
		\item As $I^* \to +\infty$:
		$$
		f(I^*) \approx \Pi - \frac{(\mu + \iota)(\mu + \gamma + d)}{\beta} + [J(\tau) - (\mu + \gamma + d)]I^* \to -\infty,
		$$
		since $J(\tau) < \mu + \gamma + d$.
	\end{itemize}
	
By the Intermediate Value Theorem, there exists a unique $I^* > 0$ satisfying $f(I^*) = 0$, which ensures the existence of a unique endemic equilibrium
 $E^*$ when $\mathcal{R}_0 > 1$.
\end{proof}
\subsection{Linearization of the equation}
We now consider the linearization of system \eqref{1.2} about the equilibrium \( \bar{r}=(\bar{S},\bar{V},\bar{I},\bar{R}(a),0) \). Initially, through the implementation of the variable transformation \( \tilde{r}(t)=r(t)-\bar{r} \) (with \( \tilde{r}(t)=(\phi_1(t),\phi_2(t),\phi_3(t),\phi_4(a,t),0) \)), we obtain
\begin{equation}\label{3.2}
\begin{cases}
	\displaystyle\frac{\mathrm{d}\tilde{r}(t)}{\mathrm{d}t}=B\tilde{r}(t)+F(\tilde{r}(t)+\bar{r})-F(\bar{r}), & t>0,\\
	\tilde{r}(0)=r(0)-\bar{r}.
\end{cases}
\end{equation}
After routine computation, the linearized system of \eqref{3.2} can be obtained as below:
$$
\begin{cases}
	\displaystyle\frac{\mathrm{d}\tilde{r}(t)}{\mathrm{d}t} = B\tilde{r}(t) + DF(\bar{r})\tilde{r}(t), & t > 0, \\
	\tilde{r}(0) = r(0) - \bar{r},
\end{cases}
$$
where
$$\begin{aligned}
	& DF(\bar{r})
	\begin{pmatrix}
		\phi_1(t) \\
		\phi_2(t) \\
		\phi_3(t) \\
		\phi_4(a,t) \\
		0
	\end{pmatrix}=
	\begin{pmatrix}
		-\beta\bar{I}\phi_1(t)-\beta\bar{S}\phi_3(t)+\eta\phi_2(t)+\int_{0}^{+\infty}\theta(a)\phi_4(a,t)da\\
	\iota\phi_1(t)-\sigma\beta\bar{I}\phi_2(t)-\sigma\beta\bar{V}\phi_3(t) \\
		\beta\bar{I}\phi_1(t)+\sigma\beta\bar{I}\phi_2(t)+\beta(\bar{S}+\sigma\bar{V})\phi_3(t) \\
		0\\
		\gamma\phi_3(t)
	\end{pmatrix}.
\end{aligned}$$
We then reformulate the Cauchy problem \eqref{3.2} as
$$\frac{\mathrm{d}\tilde{r}(t)}{\mathrm{d}t}=\tilde{B}\tilde{r}(t)+\tilde{F}(\tilde{r}(t)), \quad t \geq 0,$$
 where \( \tilde{B}=B+DF(\bar{r}) \) is a linear operator, and  
$$
\tilde{F}(\tilde{r}(t))=F(\tilde{r}(t)+\bar{r})-F(\bar{r})-DF(\bar{r})\tilde{r}(t),
$$  
this expression defines a nonlinear operator \( \tilde{F} \) that satisfies \( \tilde{F}(0)=0 \) and \( D\tilde{F}(0)=0 \).  

 Since \( \omega_{0,\text{ess}}(B_0)\leq-\chi \) and the bounded linear operator \( DF(r) \) is found to be compact. Moreover, in accordance with the perturbation conclusions in \cite{DUCROT2008501}, the following result is obtained.
\newtheorem{proposition}{Proposition}[section]  
\begin{proposition}
	As a Hille-Yosida operator, $\tilde{B}$ satisfies $\omega_{0,\mathrm{ess}}(\tilde{B}_0) \leq -\chi$ for its part $\tilde{B}_0$ acting on the subspace $X_0$.
\end{proposition}
Combining the above discussion with the result of \cite{Engel2000}, we obtain two $C_0$-semigroups $\{T_{B_0}(t)\}_{t \geq 0}$ and $\{T_{\tilde{B}_0}(t)\}_{t \geq 0}$, which are generated by their respective infinitesimal generators $B_0$ and $\tilde{B}_0$ on the space $X_0$. In addition, we obtain $\|T_{\tilde B_{0}}(t)\|\leq e^{-\xi t}$.
Note that
$$DF(\bar{r})T_{{B}_0}(t):X_0\to X,$$
which possesses the property of compactness for all $t\geq0.$ Due to
$$\begin{aligned}T_{\tilde B_{0}}(t)&=e^{DF(\bar{r})t}T_{{B}_{0}}(t)\\&=T_{{B}_{0}}(t)+\sum_{k=1}^{+\infty}\frac{(DF(\bar{r})t)^{k}}{k!}T_{{B}_{0}}(t),\end{aligned}$$
namely, $\{T_{\tilde B_{0}}(t)\}_{t\geq0}$ is a quasi-compact $C_0$-semigroup. Consequently, it is worth pointing out that $ e^{\zeta t}\|T_{\tilde B_{0}}(t)\|\to0 $ as $ t\to+\infty $ for some $ \zeta>0 $ exactly when all eigenvalues of $ \tilde{B} $ have strictly negative real parts \cite{Martcheva2003}. Inspired by the foregoing arguments, we demonstrate the result as follows.
\begin{thm}For system \eqref{1.2}, the semi-flow $T(t,r_{0})$ characterized in Theorem 2.1 yields these theoretical consequences:\\
(i) The equilibrium \( \bar{r} \) is locally asymptotically stable if all eigenvalues of \( \tilde{B} \) exhibit negative real parts.  \\
(ii)  If at least one eigenvalue of $\tilde{B}$ has a positive real part, then the equilibrium $\bar{r}$ is unstable.
\end{thm}
\section{ The Stability of Equilibria}
\subsection{ Local stability of the disease-free equilibrium}
We now analyze the local stability properties of the disease-free equilibrium $E^0=(S^0,V^0,0,0).$
\begin{thm}
	Provided that $\mathcal{R}_{0}<1$, the disease-free equilibrium $E^{0}$ is locally asymptotically stable; it becomes unstable when $\mathcal{R}_{0}>1.$
\end{thm}
\begin{proof}
By applying linearization methods, we obtain an approximate representation of the system near $E^{0} = (S^{0}, V^{0}, 0, 0)$. 
	Define  
	$$
	S(t) = S^{0} + x(t), \quad V(t) = V^{0} + y(t), \quad I(t) = z(t), \quad R(a,t) =h(a,t),
	$$  
	with \( x(t), y(t), z(t),h(a,t) \) representing small perturbations. The linearized system is expressed as follows:
\begin{equation}\label{4.1}
	\begin{cases}  
		\dfrac{dx(t)}{dt} = -\beta S^{0} z(t) - (\mu + \iota) x(t) + \eta y(t) + \int_{0}^{+\infty} \theta(a) h(a,t) da, \\  
		\dfrac{dy(t)}{dt} = \iota x(t) - \sigma \beta V^{0} z(t) - (\mu + \eta) y(t), \\  
		\dfrac{dz(t)}{dt} = \beta (S^{0} + \sigma V^{0}) z(t) - (\mu + \gamma + d) z(t), \\  
		\dfrac{\partial h(a,t)}{\partial a} + \dfrac{\partial h(a,t)}{\partial t} = -(\mu + \theta(a)) h(a,t), \\  
		h(0,t) = \gamma z(t).  
	\end{cases} 
\end{equation}	
To conduct a stability analysis, we assume exponential solutions for the configuration:		
	$$
	x(t) = \bar{x} e^{\lambda t}, \quad y(t) = \bar{y} e^{\lambda t}, \quad z(t) = \bar{z} e^{\lambda t}, \quad h(a,t) = \bar{h}(a) e^{\lambda t}.
	$$  
	Substituting these into \eqref{4.1} yields the eigenvalue problem:  
\begin{equation}\label{4.2}
	\begin{cases}  
		(\lambda + \mu + \iota) \bar{x} = -\beta S^{0} \bar{z} + \eta \bar{y} + \int_{0}^{+\infty} \theta(a) \bar{h}(a) da, \\  
		(\lambda + \mu + \eta) \bar{y} = \iota \bar{x} - \sigma \beta V^{0} \bar{z}, \\  
		(\lambda + \mu + \gamma + d) \bar{z} = \beta (S^{0} + \sigma V^{0}) \bar{z}, \\  
		\dfrac{d \bar{h}(a)}{da} = -(\lambda + \mu + \theta(a)) \bar{h}(a), \\  
		\bar{h}(0) = \gamma \bar{z}.  
	\end{cases} 
\end{equation}	
	
	Solving the fourth equation gives:  
	$$
	\bar{h}(a) = \bar{h}(0) e^{-\int_{0}^{a} (\mu + \theta(s) + \lambda) ds} = \gamma \bar{z} e^{-\int_{0}^{a} (\mu + \theta(s) + \lambda) ds}.
	$$  
	The integral term simplifies to:  
	$$
	\int_{0}^{+\infty} \theta(a) \bar{h}(a) da = \gamma \bar{z} \int_{0}^{+\infty} \theta(a) e^{-\int_{0}^{a} (\mu + \theta(s) + \lambda) ds} da = J(\lambda, \tau) \bar{z},
	$$  
	where $J(\lambda, \tau) = \dfrac{\gamma \vartheta^{*} e^{-(\lambda + \mu)\tau}}{\lambda + \mu + \vartheta^{*}}$.  
	
	Substituting back into \eqref{4.2}, the system reduces to:  
$$
	\begin{cases}  
		(\lambda + \mu + \iota) \bar{x} = -\beta S^{0} \bar{z} + \eta \bar{y} + J(\lambda, \tau) \bar{z}, \\  
		(\lambda + \mu + \eta) \bar{y} = \iota \bar{x} - \sigma \beta V^{0} \bar{z}, \\  
		(\lambda + \mu + \gamma + d) \bar{z} = \beta (S^{0} + \sigma V^{0}) \bar{z}.  
	\end{cases}
$$
	A characteristic equation emerges as the determinant of the coefficient matrix is calculated:  
\begin{equation}\label{4.3}
	g_0(\lambda) = g_{01}(\lambda) g_{02}(\lambda) = 0, 
\end{equation}
	where  
	$$\begin{aligned}
			&g_{01}(\lambda) = \beta (S^{0} + \sigma V^{0}) - (\mu + \gamma + d) - \lambda,\\
			&g_{02}(\lambda) = (\lambda + \mu + \iota)(\lambda + \mu + \eta) - \iota \eta.
	\end{aligned}$$
Considering the equation $g_{01}(\lambda)=0$, we get $\lambda_{01}=\beta(S^0+\sigma V^0)-(\mu+\gamma+d)$ and $\beta(S^0+\sigma V^0)-(\mu+\gamma+d)=(\mu+\gamma+d)(\mathcal{R}_0 -1)$. When $\mathcal{R}_0 <1$, $\lambda_{01}$ is a negative solution to Eq.~\eqref{4.3}. Consequently, the roots of this equation 
$$g_{02}(\lambda)=\lambda^2+(\iota+2\mu+\eta)\lambda+\mu(\mu+\iota+\eta)=0$$  
determine the stability of $E^0$. 

Solving this quadratic equation yields
$$\lambda_{02}=-\mu,\qquad \lambda_{03}=-\mu-\iota-\eta.$$
 Therefore, each solution to Eq.~\eqref{4.3} has a negative real part, a result that verifies \( E^0 \) is locally asymptotically stable. Conversely, when the threshold $\mathcal{R}_0 > 1$ is exceeded, we infer that $\lambda_{01}>0$. Consequently, $E^0$ grows unstable when $\mathcal{R}_0 > 1$, which concludes this proof. 
\end{proof}
\subsection{Local stability of the endemic equilibrium}
To start with, we perform a variable transformation as follows:
$$S(t) = S^* +x(t), \quad V(t) =V^*+y(t),\quad I(t) =I^*+z(t), \quad R(a,t) = R^*(a)+h(a,t).$$
Then we introduce the linearized part of the system 
\begin{equation}\label{4.4}
\begin{cases} \dfrac{dx(t)}{dt}=-\beta S^*z(t)-\beta I^* x(t)+\eta y(t)+\int_{0}^{+\infty} \theta(a) h(a,t)da,\\ 
	\dfrac{dy(t)}{dt}=\iota x(t)-\sigma\beta (V^*z(t)+I^* y(t))-(\mu+\eta)y(t),\\
	\dfrac{dz(t)}{dt}=\beta(S^* z(t)+I^* x(t))+\sigma\beta(V^* z(t)+I^* y(t))-(\mu+\gamma+d)z(t),\\	
	\dfrac{\partial{h}(a,t)}{\partial a}+\dfrac{\partial{h}(a,t)}{\partial t}=-(\mu+\theta(a))h(a,t),\\
	h(0,t)=\gamma z(t).
\end{cases}
\end{equation}
Consider the following nonzero exponential solution for the system \eqref{4.4}:	
\begin{align*}
	x(t)&= \bar{x} e^{\lambda t},\qquad y(t)=\bar{y} e^{\lambda t},\\
	z(t)&= \bar{z}e^{\lambda t},\qquad h(a,t)=\bar{h}(a) e^{\lambda t}.
\end{align*}
That is, $\bar{x}$, $\bar{y}$, $\bar{z}$, $\bar{h}(a)$, and $\lambda$ are related through the following equations:
\begin{equation}\label{4.5}
\begin{cases} 
	(\lambda+\mu+\alpha)\bar{x}=-\beta S^*\bar{z}-\beta I^*\bar{x}+\eta \bar{y}+\int_{0}^{+\infty} \theta(a)\bar{h}(a)da,\\
	(\lambda+\mu+\eta)\bar{y}=\iota\bar{x}-\sigma\beta(V^*\bar{z}+I^*\bar{y}),\\
	(\lambda+\mu+\gamma+d)\bar{z}=\beta(S^*\bar{z}+I^*\bar{x})+\sigma\beta(V^*\bar{z}+I^*\bar{y}),\\
	\dfrac{d\bar{h}(a)}{da}=-(\lambda+\mu+\theta(a))\bar{h}(a),\\
	\bar{h}(0)=\gamma \bar{z}.
\end{cases}
\end{equation}
By direct calculation, we obtain
$$\int_{0}^{+\infty} \theta(a)\bar{h}(a)da=\gamma\bar{z}\int_{0}^{+\infty} \theta(a)e^{-\int_{0}^{a}(\mu+\theta(s)+\lambda)ds}da=J(\lambda,\tau)\bar{z},$$
where $J(\lambda,\tau)=\dfrac{\gamma\vartheta^{*} e^{-(\lambda+\mu)\tau}}{\lambda+\mu+\vartheta^{*}}.$

So Eq.~\eqref{4.5} can be transformed into
\begin{equation}\label{4.6}
\begin{cases} 
	(\lambda+\mu+\iota)\bar{x}=-\beta S^*\bar{z}-\beta I^*\bar{x}+\eta \bar{y}+J(\lambda,\tau)\bar{z},\\
	(\lambda+\mu+\eta)\bar{y}=\iota\bar{x}-\sigma\beta(V^*\bar{z}+I^*\bar{y}),\\
	(\lambda+\mu+\gamma+d)\bar{z}=\beta(S^*\bar{z}+I^*\bar{x})+\sigma\beta(V^*\bar{z}+I^*\bar{y}).\\
\end{cases}
\end{equation}
Substituting these expressions into system \eqref{4.6}, we can get the characteristic equation
$$
M(\lambda,\tau)=\frac{\lambda^4+A_3\lambda^3+A_2\lambda^2+A_1\lambda+A_0+(G_1(\tau)\lambda +G_0(\tau))e^{-\lambda\tau}}{(\lambda+\mu+\iota)(\lambda+\mu+\eta)(\lambda+\mu+\gamma+d)(\lambda+\mu+\vartheta^{*})}:=\frac{m(\lambda,\tau)}{y(\lambda)},
$$
where
$$\begin{aligned}
	A_3=&[(\mu+\iota)+(\mu+\eta)+(\mu+\gamma+d)+(\mu+\vartheta^*)]-\beta(\bar{S}+\sigma\bar{V})+\sigma\beta\bar{I}+\beta\bar{I},\\
	A_2=&[(\mu+\eta)+(\mu+\gamma+d)+(\mu+\vartheta^*)]+[(\mu+\iota)(\mu+\eta)+(\mu+\iota)(\mu+\gamma+d)+(\mu+\iota)(\mu+\vartheta^*)\\
	&+(\mu+\eta)(\mu+\gamma+d)+(\mu+\eta)(\mu+\vartheta^*)+(\mu+\gamma+d)(\mu+\vartheta^*)]+\sigma^2\beta^2\bar{I}\bar{V}+\sigma\beta^2\bar{I}^2+\beta^2\bar{S}^2\\
	&+\sigma\beta\bar{I}[(\mu+\iota)+(\mu+\gamma+d)+(\mu+\vartheta^*)]-\beta(\bar{S}+\sigma\bar{V})[(\mu+\iota)+(\mu+\eta)+(\mu+\vartheta^*)]-\iota\eta\\
	&-\sigma\beta^2\bar{I}(\bar{S}+\sigma\bar{V})-\beta^2\bar{I}(\bar{S}+\sigma\bar{V}),\\
	A_1=&[(\mu+\iota)(\mu+\eta)(\mu+\gamma+d)+(\mu+\iota)(\mu+\gamma+d)(\mu+\vartheta^*)+(\mu+\eta)(\mu+\gamma+d)(\mu+\vartheta^*)\\
	&+(\mu+\iota)(\mu+\eta)(\mu+\vartheta^*)]+\beta\bar{I}[(\mu+\eta)(\mu+\gamma+d)+(\mu+\eta)(\mu+\vartheta^*)+(\mu+\gamma+d)(\mu+\vartheta^*)]\\
	&+\sigma\beta\bar{I}[(\mu+\iota)(\mu+\gamma+d)+(\mu+\iota)(\mu+\vartheta^*)+(\mu+\gamma+d)(\mu+\vartheta^*)]+\sigma^2\beta^3\bar{I}^2\bar{V}+\sigma\iota\beta^2\bar{I}\bar{S}\\
	&+\sigma^2\beta^2\bar{I}\bar{V}[(\mu+\iota)+(\mu+\vartheta^*)]+\sigma\beta^2\bar{I}^2[(\mu+\gamma+d)+(\mu+\vartheta^*)]+\beta^2\bar{S}^2[(\mu+\eta)+(\mu+\vartheta^*)]\\
	&+\iota\beta\eta(\bar{S}+\sigma\bar{V})+\sigma\eta\beta^2\bar{S}\bar{V}+\sigma\beta^3\bar{S}^2\bar{I}-\sigma\beta^2\bar{I}(\bar{S}+\sigma\bar{V})[(\mu+\iota)+(\mu+\vartheta^*)]-\sigma\beta^3\bar{I}^2(\bar{S}+\sigma\bar{V})\\
	&-\beta(\bar{S}+\sigma\bar{V})[(\mu+\iota)(\mu+\eta)+(\mu+\iota)(\mu+\vartheta^*)+(\mu+\eta)(\mu+\vartheta^*)]-\iota\eta[(\mu+\gamma+d)+(\mu+\vartheta^*)]\\
	&-\beta^2\bar{I}(\bar{S}+\sigma\bar{V})[(\mu+\eta)+(\mu+\vartheta^*)],
\end{aligned}$$
$$\begin{aligned}
	A_0=&(\mu+\iota)(\mu+\eta)(\mu+\gamma+d)(\mu+\vartheta^*)+\sigma\beta\bar{I}(\mu+\iota)(\mu+\gamma+d)(\mu+\vartheta^*)+\sigma^2\beta^2\bar{I}\bar{V}(\mu+\iota)(\mu+\vartheta^*)\\
&+\sigma\beta^2\bar{I}^2(\mu+\gamma+d)(\mu+\vartheta^*)+\iota\beta\eta(\bar{S}+\sigma\bar{V})(\mu+\vartheta^*)+\sigma^2\beta^3\bar{I}^2\bar{V}(\mu+\vartheta^*)+\sigma\iota\beta^2\bar{I}\bar{S}(\mu+\vartheta^*)\\
&+(\mu+\eta)(\mu+\gamma+d)(\mu+\vartheta^*)+\beta^2\bar{S}^2(\mu+\eta)(\mu+\vartheta^*)+\sigma\eta\beta^2\bar{S}\bar{V}(\mu+\vartheta^*)+\sigma\beta^3\bar{S}^2\bar{I}(\mu+\vartheta^*)\\
&-\beta(\bar{S}+\sigma\bar{V})(\mu+\iota)(\mu+\eta)(\mu+\vartheta^*)-\sigma\beta^2\bar{I}(\bar{S}+\sigma\bar{V})(\mu+\iota)(\mu+\vartheta^*)-\sigma\beta^3\bar{I}^2(\bar{S}+\sigma\bar{V})(\mu+\vartheta^*)\\
&-\beta^2\bar{I}(\bar{S}+\sigma\bar{V})(\mu+\eta)(\mu+\vartheta^*)-\iota\eta(\mu+\gamma+d)(\mu+\vartheta^*),\\
G_1(\tau)=&-\beta\gamma\vartheta^*\bar{I}e^{-\mu\tau},\\
G_0(\tau)=&-[(\mu+\eta)\beta\gamma\vartheta^*\bar{I}+\sigma\iota\beta\gamma\vartheta^*\bar{I}+\sigma\beta^2\gamma\vartheta^*\bar{I}^2]e^{-\mu\tau}.
\end{aligned}$$
Clearly,
$$\{\lambda\in\Omega:M(\lambda,\tau)=0\}=\{\lambda\in\Omega:m(\lambda,\tau)=0\}.$$ 
When $\tau=0$, we see
\begin{equation}\label{4.7}
m(\lambda,0)=\lambda^{4}+A_{3}\lambda^{3}+A_{2}\lambda^{2}+(A_{1}+G_{1}(0))\lambda+(A_{0}+G_{0}(0))=0.
\end{equation}
Based on the Routh–Hurwitz criterion, it can be concluded that all roots of \eqref{4.7} have negative real parts if the following conditions hold:
\begin{equation}\begin{cases} 	A_3>0,\quad A_2>0,\quad A_1+G_1(0)>0,\quad A_0+G_0(0)>0,\\	A_3A_2>A_1+G_1(0),\\	[A_3A_2-(A_1+G_1(0))](A_1+G_1(0))>A_3^2(A_0+G_0(0)).\\\end{cases}\tag{H}\label{eq:H}\end{equation} 
Consequently, the following conclusion is drawn.
  \begin{thm}
  Given that \( \mathcal{R}_0 > 1 \), \( J(\tau) < \mu + \gamma + d \), and conditions \eqref{eq:H} are fulfilled, the endemic equilibrium \( E^* = (S^*, V^*, I^*, R^*(a)) \) of system \eqref{1.2}  is locally asymptotically stable when \( \tau = 0 \).
  \end{thm}
 \section{Hopf Bifurcation}
Throughout this section, we assume that \( \mathcal{R}_0 > 1 \), \( J(\tau) < \mu + \gamma + d \), and that conditions~\eqref{eq:H} are satisfied. Under these assumptions, we examine the dynamical consequences of immune system breakdown and investigate the resulting behavior of system~\eqref{1.2} for \( \tau > 0 \).

  When $\tau>0$,
  \begin{equation}\label{5.1}
  m(\lambda,\tau)=\lambda^4+A_3\lambda^3+A_2\lambda^2+A_1\lambda+A_0+(G_1(\tau)\lambda +G_0(\tau))e^{-\lambda\tau}=0.
   \end{equation}
   Let $\lambda = i v$ with $v>0$ be a purely imaginary root of \eqref{5.1}.  Subsequently, we acquire
   $$v^4-iA_3 v^3-A_2 v^2+i A_1v+A_0+(iG_1(\tau)v+G_0(\tau))(\cos(\tau v)-i\sin(\tau v))=0.$$
    Separating the real components from the imaginary components provides
    \begin{equation}\label{5.2}
   \begin{cases} 
    v^4-A_2 v^2+A_0=-G_1(\tau)v\sin(\tau v)-G_0(\tau)\cos(\tau v),\\
    -A_3 v^3+A_1 v=-G_1(\tau)v\cos(\tau v)+G_0(\tau)\sin(\tau v).
    \end{cases}
     \end{equation}
     Squaring and adding both equations in \eqref{5.2} yields:
  \begin{equation}\label{5.3}
 v^8+(A_3^2-2A_2)v^6+(2A_0+A_2^2-2A_1A_3)v^4+(A_1^2-2A_0 A_2-G_1^2(\tau))v^2+A_0^2-G_0^2(\tau)=0.
  \end{equation}
     Let $w=v^2$, and denote $C_1=A_3^2-2A_2$, $C_2=2A_0+A_2^2-2A_1A_3$, $C_3=A_1^2-2A_0 A_2-G_1^2(\tau)$, $C_4=A_0^2-G_0^2(\tau)$. Then Eq.~\eqref{5.3} becomes
     \begin{equation}\label{5.4}
     	F(w):=w^4+C_1 w^3+C_2 w^2+C_3 w+C_4=0.
 \end{equation}     	
 \begin{lemma}
 	From $C_4<0$, it follows that Eq.~\eqref{5.4} must have one or more positive roots.
 \end{lemma}    
 \begin{proof}
 Clearly, $F(0)=C_4<0$, and $\lim_{w\to +\infty} F(w)=+\infty$. Therefore, there exists a $w_0 \in (0, +\infty)$ such that $F(w_0) = 0$. This concludes the proof.
 \end{proof}  
From the definition of $F(w)$, we have $F'(w)=4w^3+3C_1w^2+2C_2w+C_3$.\\ 
Set
\begin{equation}\label{5.5}
	4w^3+3C_1w^2+2C_2w+C_3=0.
\end{equation}
Let $w=h-\frac{3}{4}C_1$. Then Eq.~\eqref{5.5} becomes 
  $$
  	h^3+D_1 h+D_2=0,
   $$
 where
 $$D_1=\frac{1}{2}C_2-\frac{3}{16}C_1^2 ,\qquad D_2=\frac{1}{32} C_1^3-\frac{1}{8}C_1C_2+\frac{1}{4}C_3.$$  
 Define
 $$ \begin{aligned}
  	&D=\left(\frac{D_{2}}{2}\right)^{2}+\left(\frac{D_{1}}{3}\right)^{3},\quad p=\frac{-1+\sqrt{3}\mathrm{i}}{2}, \\
  	& h_{1}=\sqrt[3]{-\frac{D_{2}}{2}+\sqrt{D}}+\sqrt[3]{-\frac{D_{2}}{2}-\sqrt{D}}, \\
  	& h_{2}=\sqrt[3]{-\frac{D_{2}}{2}+\sqrt{D}}p+\sqrt[3]{-\frac{D_{2}}{2}-\sqrt{D}}p^{2}, \\
  	& h_{3}=\sqrt[3]{-\frac{D_{2}}{2}+\sqrt{D}}p^{2}+\sqrt[3]{-\frac{D_{2}}{2}-\sqrt{D}}p, \\
  	&w_{i}=h_{i}-\frac{3C_1}{4},\quad i=1,2,3.
  \end{aligned}   $$
By applying lemmas from reference \cite{LI2005519}, we have the following lemmas.
  \begin{lemma}
 Assume that $C_4\ge 0$.\\
 (i) For $D \ge 0$, Eq.~\eqref{5.4} has a positive root if and only if $w_1 > 0$ and $F(w_1) < 0$;\\
 (ii) If $D < 0$, then the existence of a positive root for Eq.~\eqref{5.4} if and only if  there exists at least one $w^* \in \{w_1, w_2, w_3\}$ such that $w^* > 0$ and $F(w^*) \le 0$.
 \end{lemma}

 Assume that Eq.~\eqref{5.4} has \( k \) positive roots, where \( k \in \{1, 2, 3, 4\} \). Denote these roots by \( w_j^* \) for \( j = 1, 2, \cdots, k \). Then Eq.~\eqref{5.3} has $k$ corresponding positive roots $v_j = \sqrt{w_j^*}$ for $j = 1, 2, \cdots, k$.
 
 From Eq.~\eqref{5.2}, we have
  \begin{equation}\label{5.6}
 \cos(\tau v)=\frac{[A_3G_1(\tau)-G_0(\tau)]v^4+[A_2 G_0(\tau)-A_1G_1(\tau)]v^2-A_0 G_0(\tau)}{G_0^2(\tau)+G_1^2(\tau)v^2}=\varphi(v).
  \end{equation}
For $v = v_j$ ($j = 1,2,\dots,k$), solving \eqref{5.6} for $\tau$ yields
 $$\tau_n^{(j)}=\frac{1}{v_{j}}[\arccos\varphi(v_{j})+2n\pi],\quad n=0,1,2,\ldots,j=1,2,\ldots k.$$    
   Note that this formula defines $\tau_n^{(j)}$ implicitly, since $\varphi$ depends on $\tau$ through $G_0(\tau)$ and $G_1(\tau)$. Numerical solution is required. Thus, when \(\tau = \tau_n^{(j)}\), \(\pm iv_j\) constitutes a pair of purely imaginary roots of Eq.~\eqref{5.1}. There is no doubt that $\left\{\tau_n^{(j)}\right\} $ forms a strictly increasing sequence for $n = 0, 1, 2,\cdots$ with $\lim_{n\to +\infty}\tau_n^{(j)}=+\infty$. Therefore, there exist $j_0\in\left\{1,2,\cdots k\right\}$ and $n_0\in\left\{0,1,2,\cdots\right\}$ such that $$\tau_{n_0}^{(j_0)}=\min\left\{\tau_n^{(j)}:j=1,2,\ldots,k;n=0,1,2,\ldots\right\}.$$
Let
  \begin{equation}\label{5.7}
  	\tau_0=\tau_{n_0}^{(j_0)},\quad v_0=v_{j_0},\quad w_0={w_{j_0}}^*.
  	 \end{equation}
 \begin{lemma}Assume condition \eqref{eq:H} holds:\\ 	
 	(i) If any condition among (a)-(c) is satisfied:
 	\begin{itemize}
 		\item[(a)] $C_4 < 0$;
 		\item[(b)] $C_4 \geq 0$, $D \geq 0$, $w_1 > 0$ and $F(w_1)\le 0$;
 		\item[(c)] $C_4 \geq 0$, $D < 0$, and there exists $w^{\ast} \in \{w_1, w_2, w_3\}$ such that $w^{\ast} > 0$ and $F(w^{\ast}) \leq 0$,
 	\end{itemize}
 	then all roots of Eq.~\eqref{5.1} possess negative real parts for $\tau \in [0, \tau_0)$;\\
 	(ii) If none of conditions (a)-(c) from (i) are fulfilled, every root of Eq.~\eqref{5.1} exhibits negative real parts throughout $\tau \geq 0$.
 \end{lemma}
 
  Let  
  \begin{equation}\label{5.8}
  s(\tau) = \xi(\tau) + iv(\tau)
   \end{equation} be a root of Eq.~\eqref{5.1} satisfying $\xi(\tau_0) = 0$, $v(\tau_0) = v_0$.
 \begin{lemma} 
 Assume that \( F'(w_j)\neq 0 \). Then for \( \tau = \tau_n^{(j)} \), Eq.~\eqref{5.1} possesses a pair of simple purely imaginary roots  $\pm i v_j$; that is,  $$\left.\frac{\mathrm{d}m(\lambda,\tau)} {\mathrm{d}\lambda}\right|_{\lambda=\mathrm{i}v_j}\neq 0.$$
  \end{lemma}
  
  \begin{proof} 
  	If $\mathrm{i}v_j$ is not simple, then
  $$
  \left.\frac{\mathrm{d}m(\lambda,\tau)}{\mathrm{d}\lambda}\right|_{\lambda=\mathrm{i}v_j}=0.$$
  Denote
  $$m_1(\lambda)=\lambda^4+A_3\lambda^3+A_2\lambda^2+A_1\lambda+A_0 ,\qquad m_2(\lambda)=G_1(\tau)\lambda+G_0(\tau).$$
  Then, we can rewrite Eq.~\eqref{5.1} as
   \begin{equation}\label{5.9}
   	m_{1}(\lambda)+m_{2}(\lambda)e^{-\lambda\tau_n^{(j)}}=0.
   	 \end{equation}
  Furthermore, we have
  $$\left\{
  \begin{array}
  	{l}\dfrac{\mathrm{d}m_1(\lambda)}{\mathrm{d}\lambda}\bigg|_{\lambda=\mathrm{i}v_j}=-4v_j^3\mathrm{i}-3A_3v_j^2+2A_2v_j\mathrm{i}+A_1, \\
  	\\
  \dfrac{\mathrm{d}m_2(\lambda)}{\mathrm{d}\lambda}\bigg|_{\lambda=\mathrm{i}v_j}=G_1(\tau).
  \end{array}\right. $$
  Thus, one yields
  \begin{equation}\label{5.10}
  \left.\frac{\mathrm{d} m_{1}(\lambda)}{\mathrm{d} \lambda}\right|_{\lambda=\mathrm{i} v_j}=-\mathrm{i} \frac{\mathrm{d}m_{1}(\mathrm{i} v_j)}{\mathrm{d} v_j}, \quad\left.\frac{\mathrm{d}m_{2}(\lambda)}{\mathrm{d} \lambda}\right|_{\lambda=\mathrm{i} v_j}=-\mathrm{i} \frac{\mathrm{d} m_{2}(\mathrm{i} v_j)}{\mathrm{d} v_j}.
\end{equation}
  Following Eq.~\eqref{5.10}, one has
  $$\begin{aligned}
  	\left.\frac{\mathrm{d} m(\lambda, \tau_n^{(j)})}{\mathrm{d} \lambda}\right|_{\lambda=\mathrm{i} v_j} &=\left.\frac{\mathrm{d}}{\mathrm{d} \lambda}\left\{m_{1}(\lambda)+m_{2}(\lambda)e^{-\lambda \tau_n^{(j)}}\right\}\right|_{\lambda=\mathrm{i} v_j} \\
  	&=-\mathrm{i} \frac{\mathrm{d} m_{1}(\mathrm{i} v_j)}{\mathrm{d} v_j}-\mathrm{i} \frac{\mathrm{d} m_{2}(\mathrm{i}v_j)}{\mathrm{d} v_j} e^{-\mathrm{i} v_j \tau_n^{(j)}}-\tau_n^{(j)} m_{2}(\mathrm{i} v_j) e^{-\mathrm{i} v_j\tau_n^{(j)}} \\
  	&=0,
  \end{aligned}$$
 which implies 
 \begin{equation}\label{5.11}
 	\frac{\mathrm{d} m_{1}(\mathrm{i} v_j)}{\mathrm{d} v_j}+\frac{\mathrm{d} m_{2}(\mathrm{i}v_j)}{\mathrm{d} v_j} e^{-\mathrm{i} v_j \tau_n^{(j)}}-\mathrm{i}\tau_n^{(j)} m_{2}(\mathrm{i} v_j) e^{-\mathrm{i} v_j\tau_n^{(j)}}=0.
 	 \end{equation}
  It follows from Eqs.~\eqref{5.9} and \eqref{5.11} that one derives
  $$\mathrm{i}\tau_n^{(j)}=\frac{1}{m_2(\mathrm{i}v_j)}\frac{\mathrm{d}m_2(\mathrm{i}v_j)}{\mathrm{d}v_j}-\frac{1}{m_1(\mathrm{i}v_j)}\frac{\mathrm{d}m_1(\mathrm{i}v_j)}{\mathrm{d}v_j},$$
and
$$\begin{aligned}
	\mathrm{Re}(\mathrm{i}\tau_n^{(j)})= & \mathrm{Re}\left\{\frac{1}{m_{2}(\mathrm{i}v_j)}\frac{\mathrm{d}m_{2}(\mathrm{i}v_j)}{\mathrm{d}v_j}-\frac{1}{m_{1}(\mathrm{i}v_j)}\frac{\mathrm{d}m_{1}(\mathrm{i}v_j)}{\mathrm{d}v_j}\right\} \\
	= & \mathrm{Re}\left\{\frac{\overline{m_{2}(\mathrm{i}v_j)}}{m_{2}(\mathrm{i}v_j)\overline{m_{2}(\mathrm{i}v_j)}}\frac{\mathrm{d}m_{2}(\mathrm{i}v_j)}{\mathrm{d}v_j}-\frac{\overline{m_{1}(\mathrm{i}v_j)}}{m_{1}(\mathrm{i}v_j)\overline{m_{1}(\mathrm{i}v_j)}}\frac{\mathrm{d}m_{1}(\mathrm{i}v_j)}{\mathrm{d}v_j}\right\}.
\end{aligned}$$  
 Since $m_{1}(i v_j)=-m_{2}(i v_j)e^{-i v_j\tau_n^{(j)}}$, we have $|m_1(\mathrm{i}v_j)|=|m_2(\mathrm{i}v_j)|$. we further derive
  $$\begin{aligned}
  	\mathrm{Re}(\mathrm{i}\tau_n^{(j)})=& \mathrm{Re}\left\{\frac{1}{m_{1}(\mathrm{i}v_j)\overline{m_{1}(\mathrm{i}v_j)}}\left(\overline{m_{2}(\mathrm{i}v_j)}\frac{\mathrm{d}m_{2}(\mathrm{i}v_j)}{\mathrm{d}v_j}-\overline{m_{1}(\mathrm{i}v_j)}\frac{\mathrm{d}m_{1}(\mathrm{i}v_j)}{\mathrm{d}v_j}\right)\right\} \\
  	=&-\frac{v_j}{|m_{1}(\mathrm{i}v_j)|^{2}}\left[4w_j^{3}+(3A_{3}^{2}-6A_{2})w_j^{2}\right] \\
  	& -\frac{v_j}{|m_{1}(\mathrm{i}v_j)|^{2}}\left[(2A_{2}^{2}+4A_{0}-4A_{3}A_{1})w_j+(A_{1}^{2}-2A_{2}A_{0}-G_1(\tau)^{2})\right] \\
  	=&-\frac{v_j}{\lvert m_{1}(\mathrm{i}v_j)\rvert^{2}} F'(w_j) \neq 0.
  \end{aligned}$$
  The equality $\mathrm{Re}(\mathrm{i}\tau_n^{(j)}) \neq 0$ contradicts the real-valued nature of $\tau_n^{(j)}$, thereby concluding the proof.
  \end{proof}
  \begin{lemma}
  Within the framework of Eq.~\eqref{5.4}, assuming $w_j = v_j^2$ with $F'(w_j) \neq 0$ yields the sign correspondence:
  	$$\mathrm{sign}\left\{\left.\frac{\mathrm{d}(\mathrm{Re}\lambda)}{\mathrm{d}\tau}\right|_{\tau=\tau_{n}^{(j)}}\right\}= \mathrm{sign}\left\{F'(w_j)\right\}.$$
  \end{lemma} 
 \begin{proof}
  Differentiating Eq.~\eqref{5.1} with respect to $\tau$, we have  
  $$(4\lambda^3+3A_3\lambda^2+2A_2\lambda+A_1)\frac{\mathrm{d}\lambda}{\mathrm{d}\tau}+\left(G_1(\tau)-G_1(\tau)\mu\lambda-G_0(\tau)\mu\right)e^{-\lambda\tau}\frac{\mathrm{d}\lambda}{\mathrm{d}\tau}=\left(G_1(\tau)\lambda+G_0(\tau)\right)\lambda e^{-\lambda\tau}.$$
   This gives
  $$\begin{aligned}
  	\left(\frac{\mathrm{d}\lambda}{\mathrm{d}\tau}\right)^{-1} & =-\frac{4\lambda^3+3A_3\lambda^2+2A_2\lambda+A_1}{(G_1(\tau)\lambda+G_0(\tau))\lambda e^{-\lambda\tau}}  +\frac{G_1(\tau)}{\lambda(G_1(\tau)\lambda+G_0(\tau))}-\frac{\mu}{\lambda}\\&=-\frac{4\lambda^3+3A_3\lambda^2+2A_2\lambda+A_1}{\lambda(\lambda^4+A_3\lambda^3+A_2\lambda^2+A_1\lambda+A_0)}  +\frac{G_1(\tau)}{\lambda(G_1(\tau)\lambda+G_0(\tau))}-\frac{\mu}{\lambda}.
  \end{aligned}$$
   Furthermore, there holds
   $$\mathrm{sign}\left\{\left.\frac{\mathrm{d}(\mathrm{Re}\lambda)}{\mathrm{d}\tau}\right|_{\tau=\tau_{n}^{(j)}}\right\}=\mathrm{sign}\left\{\mathrm{Re}\left(\frac{\mathrm{d}\lambda}{\mathrm{d}\tau}\right)^{-1}\bigg|_{\lambda=iv_j}\right\}$$
  
   \[
   \begin{aligned}
   	&= \mathrm{sign}\left\{
   	\frac{4v_j^{6}+3(A_{3}^{2}-2A_{2})v_j^{4}+2(2A_{0}+A_{2}^{2}-2A_{1}A_{3})v_j^{2}+(A_{1}^{2}-2A_{0}A_{2})}
   	{(v_j^{4}-A_{2}v_j^{2}+A_{0})^{2} + (A_{1}v_j-A_{3}v_j^{3})^{2}} \right.\\
   	&\quad \left.- \frac{G_{1}^{2}(\tau)}{G_{1}^{2}(\tau)v_j^{2}+G_{0}^{2}(\tau)}
   	\right\}.
   \end{aligned}
   \]
    From Eq.~\eqref{5.3}, we get
    $$\left(v_j^{4}-A_{2}v_j^{2}+A_{0}\right)^{2}+\left(A_{1}v_j-A_{3}v_j^{3}\right)^{2}=G_1^2(\tau)v_j^2+G_0^2(\tau).$$
    Thus, we obtain
    $$\begin{aligned}
    \mathrm{sign}\left\{	\left.\frac{\mathrm{d}(\mathrm{Re}\lambda)}{\mathrm{d}\tau}\right|_{\tau=\tau_{n}^{(j)}}\right\}= &\mathrm{sign} \left\{\frac{4v_j^{6}+3C_1v_j^{4}+2C_2v_j^{2}+C_3}{\left(v_j^{4}-A_{2}v_j^{2}+A_{0}\right)^{2}+\left(A_{1}v_j-A_{3}v_j^{3}\right)^{2}}\right\}\\
    	= &\mathrm{sign}\left\{ \frac{F^{\prime}(w_j)}{\left(v_j^{4}-A_{2}v_j^{2}+A_{0}\right)^{2}+\left(A_{1}v_j-A_{3}v_j^{3}\right)^{2}}\right\}.
    \end{aligned}$$
    Hence, there holds $\mathrm{sign}\left\{\left.\frac{\mathrm{d}(\mathrm{Re}\lambda)}{\mathrm{d}\tau}\right|_{\tau=\tau_{n}^{(j)}}\right\}= \mathrm{sign}\left\{F'(w_j)\right\}$, that is, the sign of $\left.\frac{\mathrm{d}(\mathrm{Re}\lambda)}{\mathrm{d}\tau}\right|_{\tau=\tau_{n}^{(j)}}$ is the same as $F'(w_j)$.
 \end{proof}
The following theorem is derived from the results of Lemmas 5.1–5.5.
   \begin{thm}
 Adopting the definitions \eqref{5.7} and \eqref{5.8} for $v_0$, $w_0$, $\tau_0$, and $s(\tau)$, and assuming that \eqref{eq:H} holds, we have:
 \begin{itemize}
 	\item[(i)] If none of the conditions (a)–(c) in Lemma~5.3 are satisfied, then the equilibrium $E^*$ is locally asymptotically stable for all $\tau \ge 0$;
 	\item[(ii)] If any one of the conditions (a)–(c) is satisfied, then $E^*$ is locally asymptotically stable for all $\tau \in [0,\tau_0)$;
 	\item[(iii)] Moreover, if one of the conditions (a)–(c) holds and $F'(w_0)\ne 0$, then a Hopf bifurcation occurs for system~\eqref{1.2} at $E^*$ as $\tau$ passes through $\tau_0$.
 \end{itemize}
  \end{thm}
 \section{Numerical simulations}
In this part, we apply  the forward difference approach to discretize the system \eqref{1.2} and subsequently perform numerical simulations to validate our theoretical findings. For the purpose of conducting subsequent analyses, we fix the maximum recovery age at 100 and use the following parameter values:  	$$  	\mu = 0.008,\ \beta = 0.0009,\ \eta = 0.01,\ \iota = 0.58,\ \gamma = 0.45,\ d = 0.05,\ \sigma = 0.5,\ \vartheta^* = 0.35.  	$$  	

The initial conditions are chosen as 	$$  	S_0 = 100,\quad V_0 = 50,\quad I_0 = 20,\quad R_0(a) = 10e^{-0.05a}.  	$$  	  	

First, when $\Pi = 5$ and $\tau = 12$, we obtain $\mathcal{R}_0 = 0.57031 < 1$, 
and the disease-free equilibrium $E^0$ is locally asymptotically stable. Figure 2 illustrates time-dependent changes for each component of the system \eqref{1.2}.  	  	

Proceeding with $\Pi$ set to 20, we determine $\mathcal{R}_0 = 2.2812 > 1$. With $\tau = 12$, $J(\tau) = 0.39967$ and $\mu + \gamma + d = 0.508$. Since $J(\tau) < \mu + \gamma + d$, Figure 3 demonstrates that the endemic equilibrium $E^*$ is locally asymptotically stable. When $\tau$ is increased to 19, $J(\tau) = 0.37791 < \mu + \gamma + d = 0.508$, and Figure 4 confirms this stability.  	  	

Finally, when $\tau$ is further increased to 20, we obtain $J(\tau) = 0.3749 < \mu + \gamma + d = 0.508$. Figure 5 illustrates the loss of stability of $E^*$ through Hopf bifurcation, resulting in the onset of periodic dynamics. For $\tau = 60$, where $J(\tau) = 0.2723 < \mu + \gamma + d = 0.508$, numerical simulations show sustained periodic oscillations in the system (see Figure 6). These results indicate that when $\tau$ is nonzero, the Hopf bifurcation threshold $\tau_0$ lies in the interval (19, 20).
  	
\begin{figure}[htbp] % [htbp] 是位置参数，表示这里、顶部、底部或单独一页
	\centering
	% 第一行，三张子图
	\begin{subfigure}{0.3\textwidth}
		\centering
		\includegraphics[width=4.2cm,height=3.15cm]{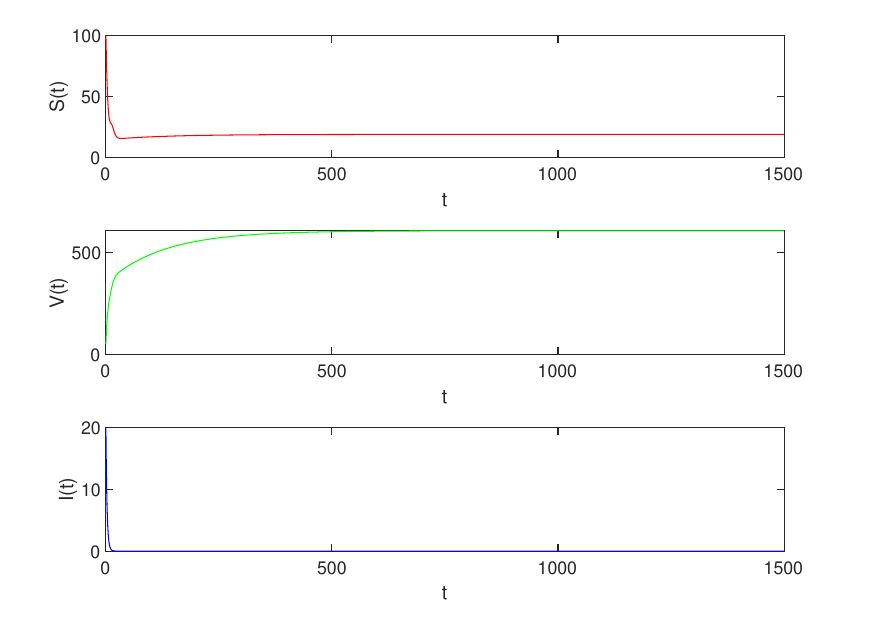}
		\caption{}
		\label{fig:image1}
	\end{subfigure}%
	\hfill
	\begin{subfigure}{0.3\textwidth}
		\centering
		\includegraphics[width=4.2cm,height=3.15cm]{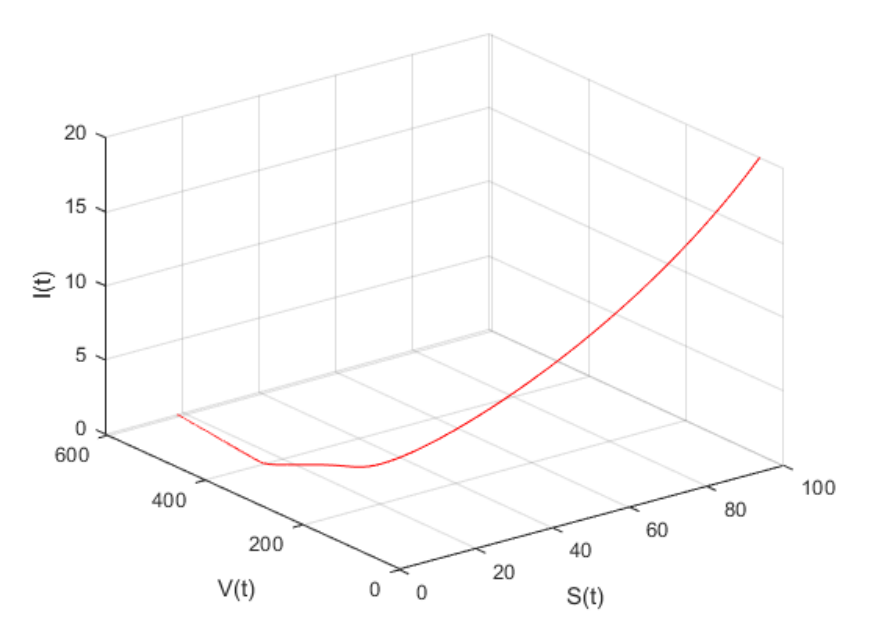}
		\caption{}
		\label{fig:image2}
	\end{subfigure}%
			\hfill
		\begin{subfigure}{0.3\textwidth}
			\centering
			\includegraphics[width=4.2cm,height=3.15cm]{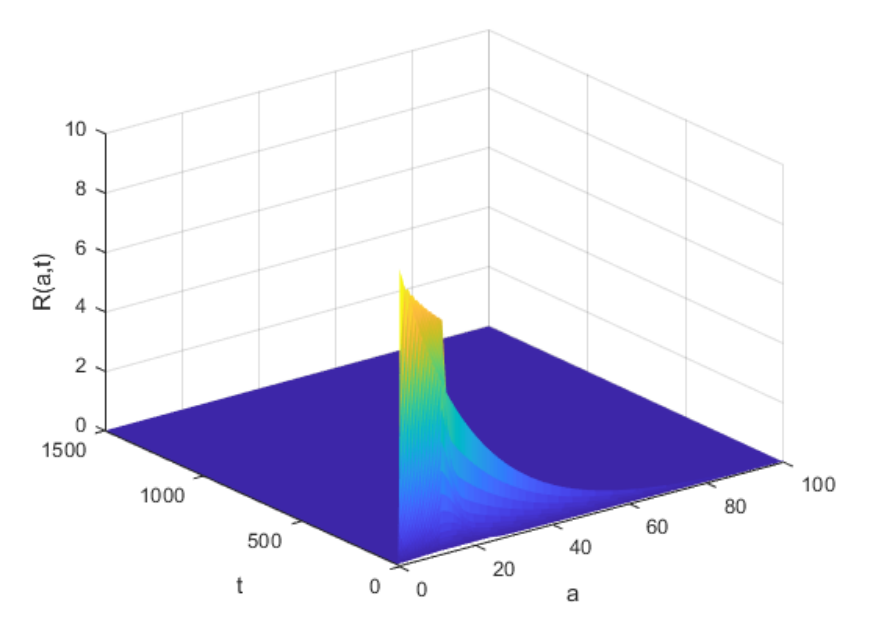}
			\caption{}
			\label{fig:image2}
		\end{subfigure}%
	\caption{ When $\Pi=5,\tau=12$, the disease-free equilibrium $E^0$ is locally asymptotically stable (LAS) for $\mathcal{R}_0 < 1$}
	\label{fig}
\end{figure}
\begin{figure}[htbp] % [htbp] 是位置参数，表示这里、顶部、底部或单独一页
	\centering
	% 第一行，三张子图
	\begin{subfigure}{0.3\textwidth}
		\centering
		\includegraphics[width=4.2cm,height=3.15cm]{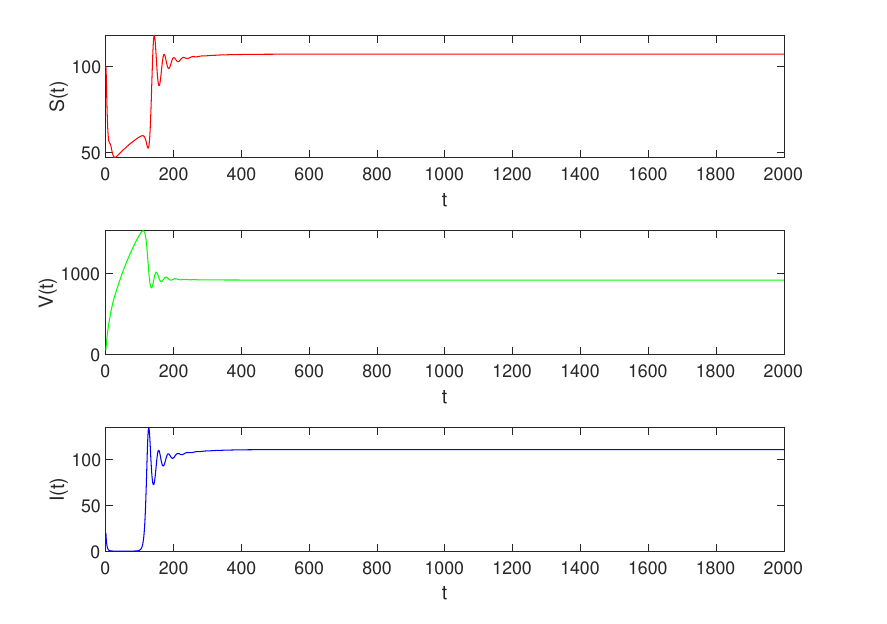}
		\caption{}
		\label{fig:image1}
	\end{subfigure}%
	\hfill
	\begin{subfigure}{0.3\textwidth}
		\centering
		\includegraphics[width=4.2cm,height=3.15cm]{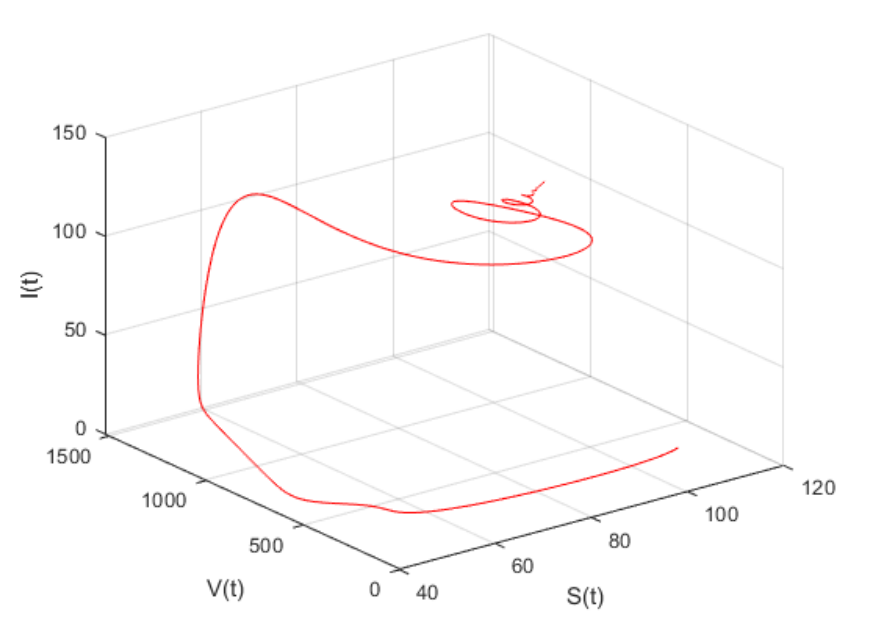}
		\caption{}
		\label{fig:image2}
	\end{subfigure}%
	\hfill
	\begin{subfigure}{0.3\textwidth}
		\centering
		\includegraphics[width=4.2cm,height=3.15cm]{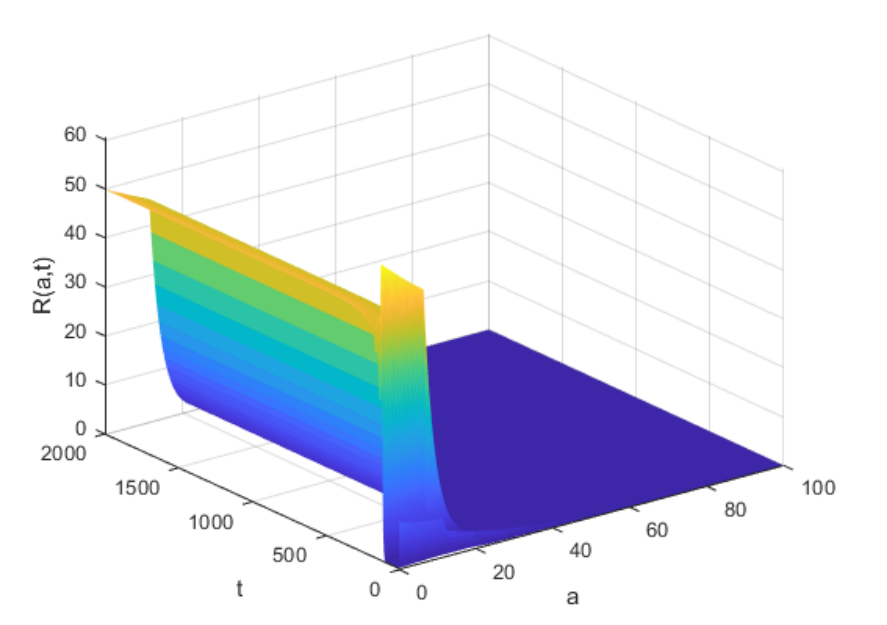}
		\caption{}
		\label{fig:image2}
	\end{subfigure}%
	\caption{ When $\Pi=20,\tau=12$, the endemic equilibrium $E^*$ is LAS for $\mathcal{R}_0 > 1$}
	\label{fig}
\end{figure}
\begin{figure}[htbp] % [htbp] 是位置参数，表示这里、顶部、底部或单独一页
	\centering
	% 第一行，三张子图
	\begin{subfigure}{0.3\textwidth}
		\centering
		\includegraphics[width=4.2cm,height=3.15cm]{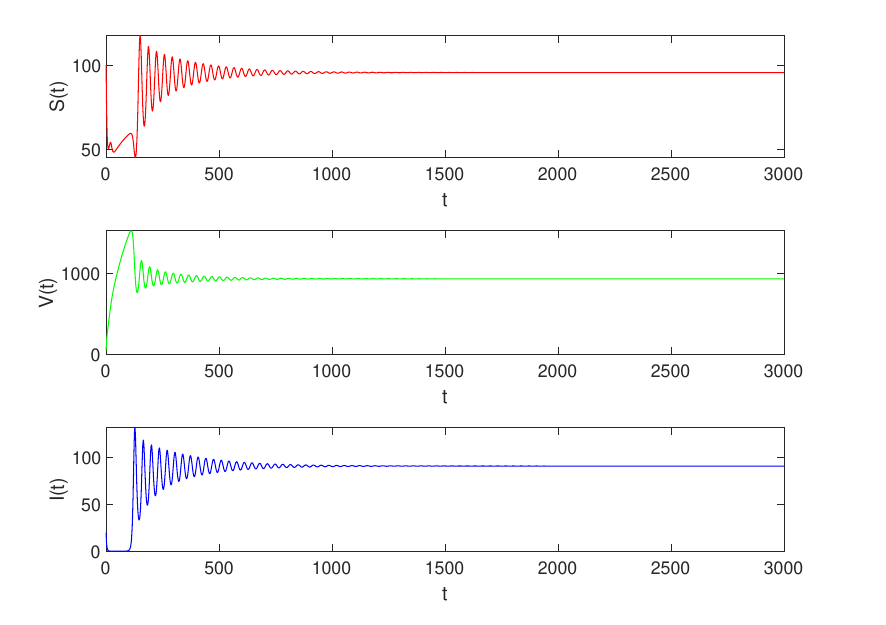}
		\caption{}
		\label{fig:image1}
	\end{subfigure}%
	\hfill
	\begin{subfigure}{0.3\textwidth}
		\centering
		\includegraphics[width=4.2cm,height=3.15cm]{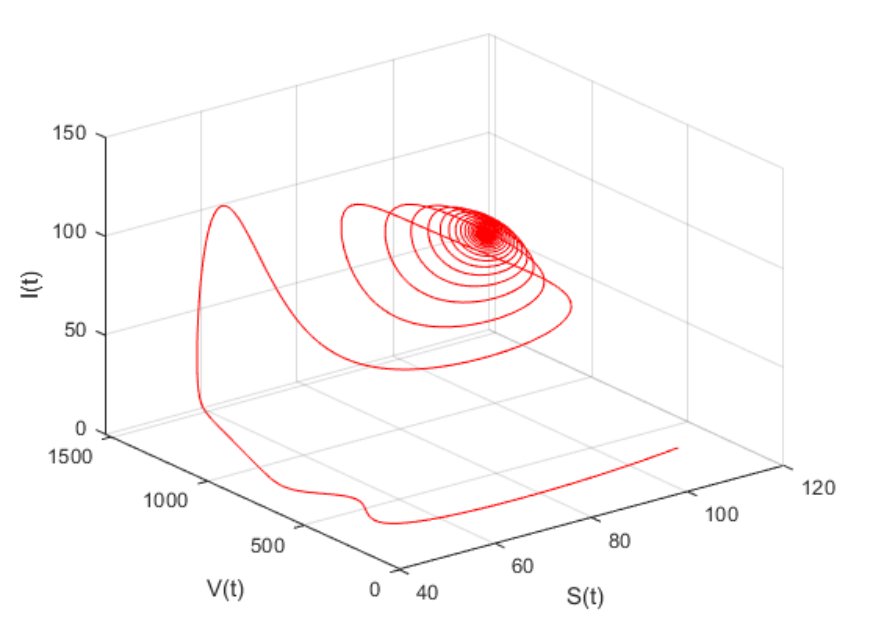}
		\caption{}
		\label{fig:image2}
	\end{subfigure}%
	\hfill
	\begin{subfigure}{0.3\textwidth}
		\centering
		\includegraphics[width=4.2cm,height=3.15cm]{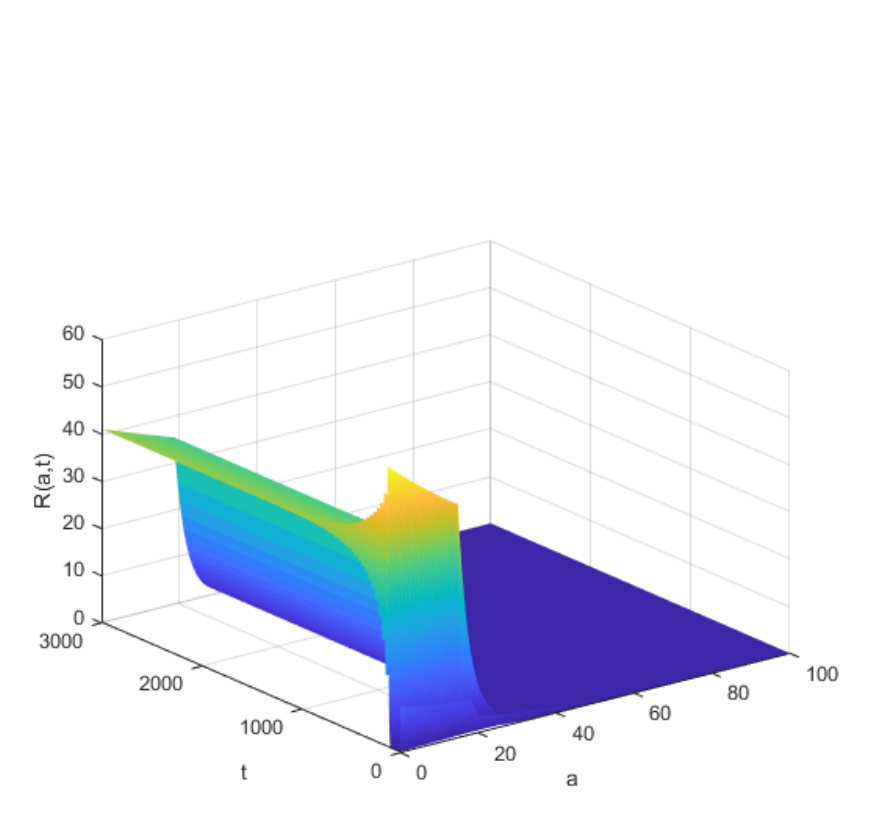}
		\caption{}
		\label{fig:image2}
	\end{subfigure}%
	\caption{ When $\Pi=20,\tau=19$, the endemic equilibrium $E^*$ is LAS for $\mathcal{R}_0 > 1$}
	\label{fig}
\end{figure}
\begin{figure}[htbp] % [htbp] 是位置参数，表示这里、顶部、底部或单独一页
	\centering
	% 第一行，三张子图
	\begin{subfigure}{0.3\textwidth}
		\centering
		\includegraphics[width=4.2cm,height=3.15cm]{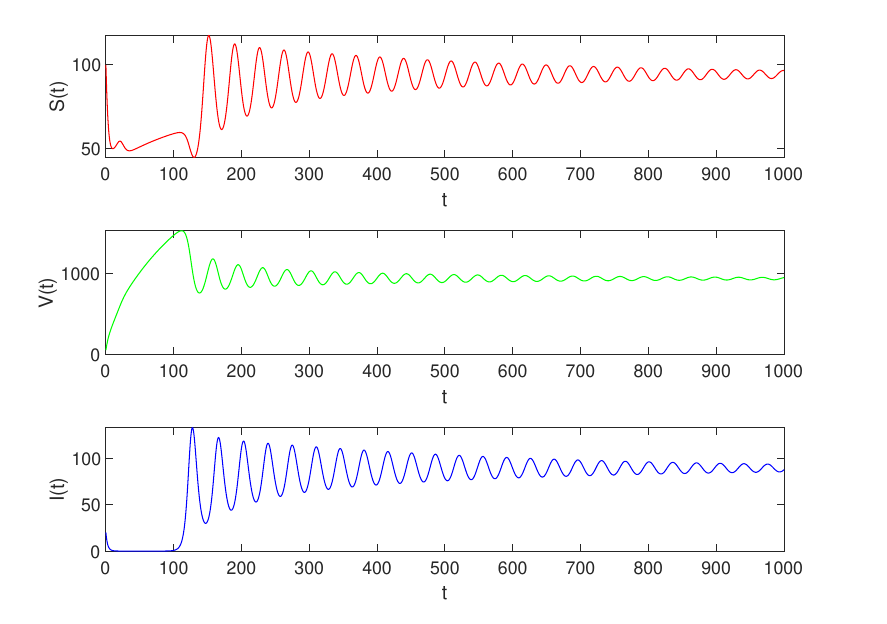}
		\caption{}
		\label{fig:image1}
	\end{subfigure}%
	\hfill
	\begin{subfigure}{0.3\textwidth}
		\centering
		\includegraphics[width=4.2cm,height=3.15cm]{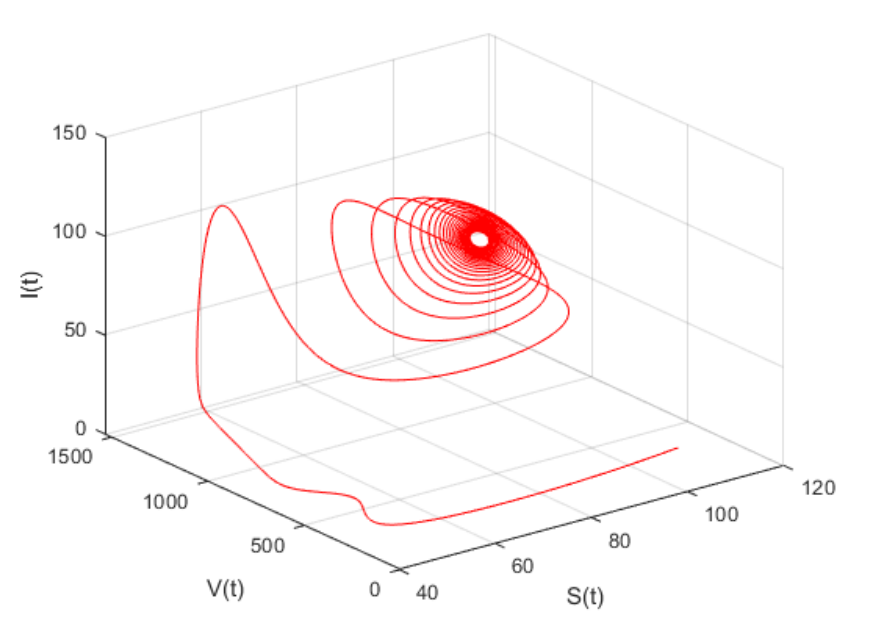}
		\caption{}
		\label{fig:image2}
	\end{subfigure}%
	\hfill
	\begin{subfigure}{0.3\textwidth}
		\centering
		\includegraphics[width=4.2cm,height=3.15cm]{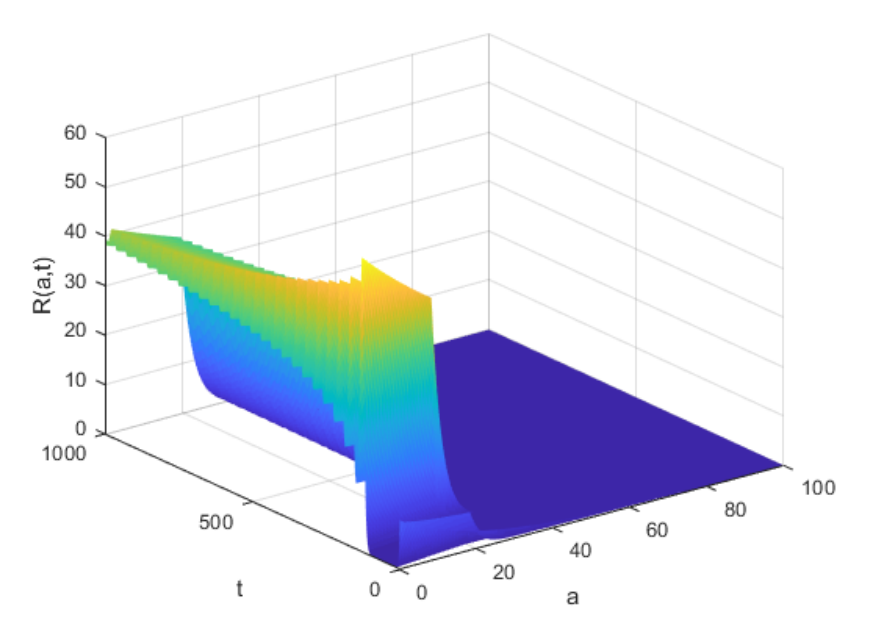}
		\caption{}
		\label{fig:image2}
	\end{subfigure}%
	\caption{When $\Pi=20,\tau=20$, the endemic equilibrium $E^*$ is not LAS for $\mathcal{R}_0 > 1$}
	\label{fig}
\end{figure}
\begin{figure}[htbp] % [htbp] 是位置参数，表示这里、顶部、底部或单独一页
	\centering
	% 第一行，三张子图
	\begin{subfigure}{0.3\textwidth}
		\centering
		\includegraphics[width=4.2cm,height=3.15cm]{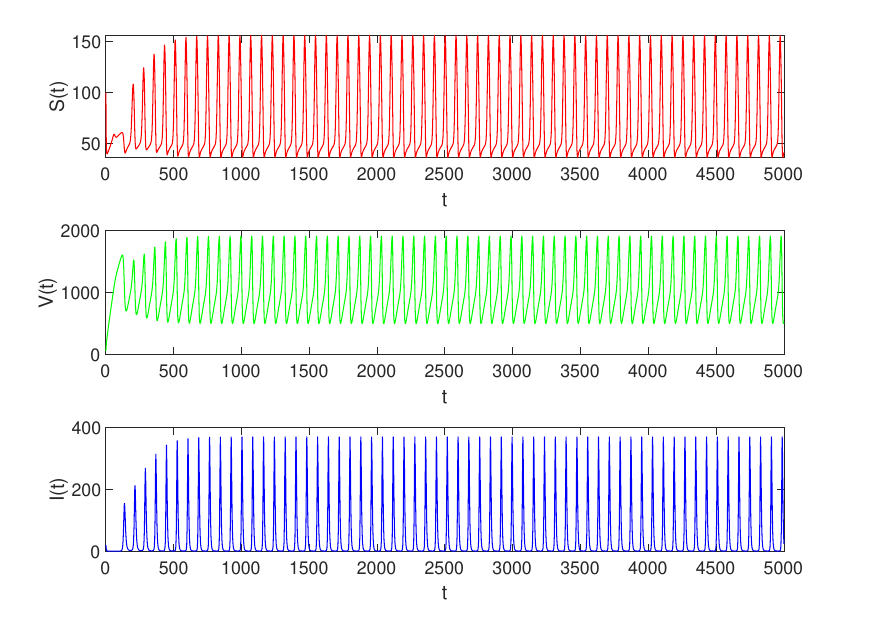}
		\caption{}
		\label{fig:image1}
	\end{subfigure}%
	\hfill
	\begin{subfigure}{0.3\textwidth}
		\centering
		\includegraphics[width=4.2cm,height=3.15cm]{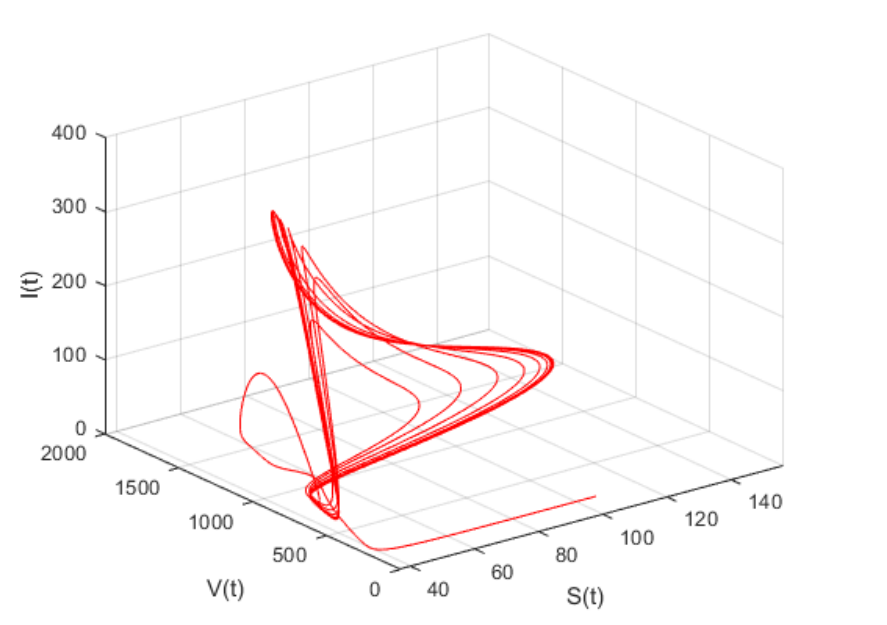}
		\caption{}
		\label{fig:image2}
	\end{subfigure}%
	\hfill
	\begin{subfigure}{0.3\textwidth}
		\centering
		\includegraphics[width=4.2cm,height=3.15cm]{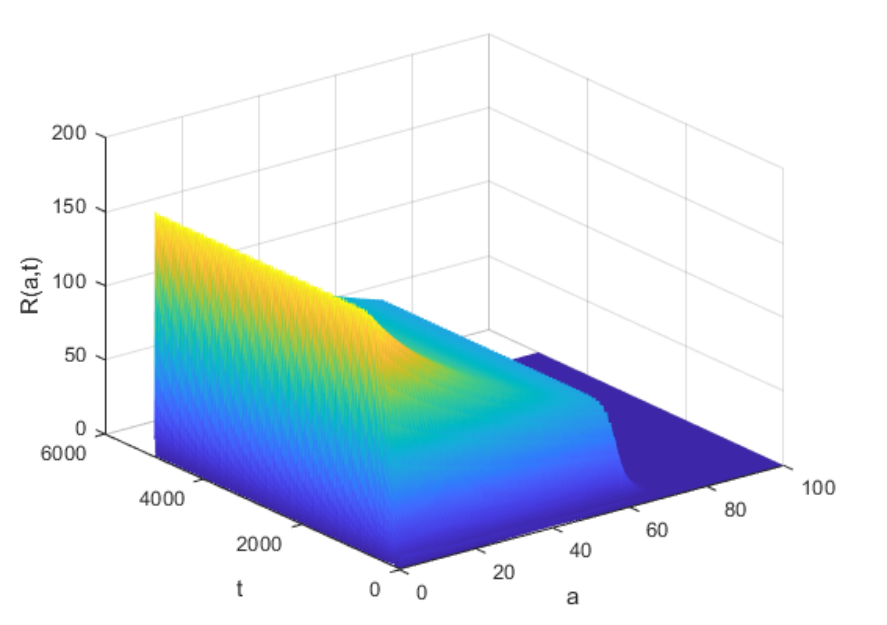}
		\caption{}
		\label{fig:image2}
	\end{subfigure}%
	\caption{ When $\Pi=20,\tau=60$, the endemic equilibrium $E^*$ is not LAS for $\mathcal{R}_0 > 1$}
	\label{fig}
\end{figure}
\section{Conclusion}
To investigate the impact of post-vaccination immune persistence on disease transmission, this paper constructs an SVIRS infectious disease model integrating recovery age structure and temporary immunity. First, we transform system \eqref{1.2} into a non-densely defined abstract Cauchy problem. Through Hille-Yosida operator theory, we verify that all solutions possess the essential properties of existence, uniqueness, non-negativity, and boundedness. Subsequently, from equilibrium analysis emerges the basic reproduction number $\mathcal{R}_0$, possessing clear epidemiological significance. Results show global asymptotic stability of the disease-free equilibrium when $\mathcal{R}_0<1$, while local asymptotic stability of the endemic equilibrium emerges when $\mathcal{R}_0>1$ under specific conditions.

Numerical simulations reveal that the immunity period $\tau$ plays a determining role in system stability. Specifically, when $\tau$ surpasses the critical threshold $\tau_0$, Hopf bifurcation emerges near the positive steady-state of the system. This manifests as the system remaining stable with short-term immunity ($\tau < \tau_0$), while generating periodic oscillations with long-term immunity ($\tau > \tau_0$). These results establish a theoretical foundation for comprehending how immune duration influences disease transmission dynamics, while simultaneously providing practical epidemic control strategies.

Considering the aforementioned research findings, we propose that the impact of immune duration should be fully considered in practical epidemic prevention and control: maintaining the stability of the system heavily relies on the timing of booster vaccine administration. Meanwhile, by appropriately regulating key parameters such as the immune decay rate $\vartheta^*$ and vaccine efficacy $\sigma$, disease control strategies can be effectively optimized.

\vskip 20 pt
\noindent{\bf Acknowledgement}
\vskip 10 pt
The first author is partially supported by the National Key Research and Development Program of China (Grant No. 2020YFA0713100).
\nocite{*}
\bibliography{t} 

\end{document}